\documentclass[10pt]{amsart}
\usepackage{amsmath, amsthm, amsfonts, amssymb, graphicx, bm,verbatim,mathrsfs,siunitx}
\usepackage{hyperref}
\usepackage{stmaryrd,enumerate,tikz}
\theoremstyle{plain}
\newtheorem{thrm}{Theorem}[section]
\newtheorem{lmm}[thrm]{Lemma}
\newtheorem{prpstn}[thrm]{Proposition}

\newtheorem{dfntn}{Definition}

\newtheorem*{rmk}{Remark}

\newtheorem{nthrm}{Theorem}

\numberwithin{equation}{section}

\newcommand{\Mod}[1]{\ (\mathrm{mod}\ #1)}

\renewcommand{\hat}[1]{\widehat #1}
\usepackage[title]{appendix}

\begin{document}

\title[Primes in arithmetic progressions to large moduli II]{Primes in arithmetic progressions to large moduli II: Well-factorable estimates}
\author{James Maynard}
\address{Mathematical Institute, Radcliffe Observatory quarter, Woodstock Road, Oxford OX2 6GG, England}
\email{james.alexander.maynard@gmail.com}
\begin{abstract}
We establish new mean value theorems for primes of size $x$ in arithmetic progressions to moduli as large as $x^{3/5-\epsilon}$ when summed with suitably well-factorable weights. This extends well-known work of Bombieri, Friedlander and Iwaniec, who handled moduli of size at most $x^{4/7-\epsilon}$. This has consequences for the level of distribution for sieve weights coming from the linear sieve.
\end{abstract}
%\newgeometry{bottom=0.1cm}
%\begin{samepage}
\maketitle
%\tableofcontents
%\end{samepage}
%\restoregeometry
\parskip 7.2pt
%
%
%
%
%%%%%%%%%%%%%%%%%%%%%%%%%%%%%%%%%%%%%%%%%%%%%%%%%%%%%%%%%%
%
%
%
%
\section{Introduction}\label{sec:Introduction}
%
%
%
%
%%%%%%%%%%%%%%%%%%%%%%%%%%%%%%%%%%%%%%%%%%%%%%%%%%%%%%%%%%
%
%
%
%
The Bombieri-Vinogradov Theorem \cite{Bombieri,Vinogradov} states that for every $A>0$ and $B=B(A)$ sufficiently large in terms of $A$ we have
\begin{equation}
\sum_{q\le x^{1/2}/(\log{x})^B}\sup_{(a,q)=1}\Bigl|\pi(x;q,a)-\frac{\pi(x)}{\phi(q)}\Bigr|\ll_A\frac{x}{(\log{x})^A},
\label{eq:BV}
\end{equation}
where $\pi(x)$ is the number of primes less than $x$, and $\pi(x;q,a)$ is the number of primes less than $x$ congruent to $a\Mod{q}$. This implies that primes of size $x$ are roughly equidistributed in residue classes to moduli of size up to $x^{1/2-\epsilon}$, on average over the moduli. For many applications in analytic number theory (particularly sieve methods) this estimate is very important, and serves as an adequate substitute for the Generalized Riemann Hypothesis (which would imply a similar statement for each \textit{individual} arithmetic progression).

We believe that one should be able to improve \eqref{eq:BV} to allow for larger moduli, but unfortunately we do not know how to establish \eqref{eq:BV} with the summation extended to $q\le x^{1/2+\delta}$ for any fixed $\delta>0$. The Elliott-Halberstam Conjecture \cite{ElliottHalberstam} is the strongest statement of this type, and asserts that for any $\epsilon,A>0$
\begin{equation}
\sum_{q\le x^{1-\epsilon} }\sup_{(a,q)=1}\Bigl|\pi(x;q,a)-\frac{\pi(x)}{\phi(q)}\Bigr|\ll_{\epsilon,A}\frac{x}{(\log{x})^A}.
\label{eq:EH}
\end{equation}
Quantitatively stronger variants of \eqref{eq:BV} such as \eqref{eq:EH} would naturally give quantitatively stronger estimates of various quantities in analytic number theory relying on \eqref{eq:BV}.

In many applications, particularly those coming from sieve methods, one does not quite need to have the full strength of an estimate of the type \eqref{eq:BV}. It is often sufficient to measure the difference between $\pi(x;q,a)$ and $\pi(x)/\phi(q)$ only for a fixed bounded integer $a$ (such as $a=1$ or $a=2$) rather than taking the worst residue class in each arithmetic progression. Moreover, it is also often sufficient to measure the difference between $\pi(x;q,a)$ and $\pi(x)/\phi(q)$ with `well-factorable' weights (which naturally appear in sieve problems) rather than absolute values. With these technical weakenings we \textit{can} produce estimates analogous to \eqref{eq:BV} which involve moduli larger than $x^{1/2}$. Formally, we define `well-factorable' weights as follows.
%
%
%
%
%%%%%%%%%%%%%%%%%%%%%%%%%%%%%%%%%%%%%%%%%%%%%%%%%%%%%%%%%%
%
%
%
%
\begin{dfntn}[Well factorable]
Let $Q\in\mathbb{R}$. We say a sequence $\lambda_q$ is \textbf{well factorable of level $Q$} if, for any choice of factorization $Q=Q_1Q_2$ with $Q_1,Q_2\ge 1$, there exist two sequences $\gamma^{(1)}_{q_1},\gamma^{(2)}_{q_2}$ such that:
\begin{enumerate}
\item $|\gamma^{(1)}_{q_1}|,|\gamma^{(2)}_{q_2}|\le 1$ for all $q_1,q_2$.
\item $\gamma^{(i)}_{q}$ is supported on $1\le q\le Q_i$ for $i\in\{1,2\}$.
\item We have
\[
\lambda_q=\sum_{q=q_1q_2}\gamma^{(1)}_{q_1}\gamma^{(2)}_{q_2}.
\]
\end{enumerate}
\end{dfntn}
%
%
%
%
%%%%%%%%%%%%%%%%%%%%%%%%%%%%%%%%%%%%%%%%%%%%%%%%%%%%%%%%%%
%
%
%
%
The following celebrated result of Bombieri-Friedlander-Iwaniec \cite[Theorem 10]{BFI1} then gives a bound allowing for moduli as large as $x^{4/7-\epsilon}$ in this setting.
%
%
%
%
%%%%%%%%%%%%%%%%%%%%%%%%%%%%%%%%%%%%%%%%%%%%%%%%%%%%%%%%%%
%
%
%
%
\begin{nthrm}[Bombieri, Friedlander, Iwaniec]\label{nthrm:BFI1}
Let $a\in\mathbb{Z}$ and $A,\epsilon>0$. Let $\lambda_q$ be a sequence which is well-factorable of level $Q\le x^{4/7-\epsilon}$. Then we have
\[
\sum_{\substack{q\le Q\\(q,a)=1}}\lambda_q\Bigl(\pi(x;q,a)-\frac{\pi(x)}{\phi(q)}\Bigr)\ll_{a,A,\epsilon}\frac{x}{(\log{x})^A}.
\]
\end{nthrm}
%
%
%
%
%%%%%%%%%%%%%%%%%%%%%%%%%%%%%%%%%%%%%%%%%%%%%%%%%%%%%%%%%%
%
%
%
%
 In this paper we consider weights satisfying a slightly stronger condition of being `triply well factorable'. For these weights we can improve on the range of moduli.
%
%
%
%
%%%%%%%%%%%%%%%%%%%%%%%%%%%%%%%%%%%%%%%%%%%%%%%%%%%%%%%%%%
%
%
%
%
\begin{dfntn}[Triply well factorable]\label{dfntn:WellFactorable}
Let $Q\in\mathbb{R}$. We say a sequence $\lambda_q$ is \textbf{triply well factorable of level $Q$} if, for any choice of factorization $Q=Q_1Q_2Q_3$ with $Q_1,Q_2,Q_3\ge 1$, there exist three sequences $\gamma^{(1)}_{q_1},\gamma^{(2)}_{q_2},\gamma^{(3)}_{q_3}$ such that:
\begin{enumerate}
\item $|\gamma^{(1)}_{q_1}|,|\gamma^{(2)}_{q_2}|,|\gamma^{(3)}_{q_3}|\le 1$ for all $q_1,q_2,q_3$.
\item $\gamma^{(i)}_{q}$ is supported on $1\le q\le Q_i$ for $i\in\{1,2,3\}$.
\item We have
\[
\lambda_q=\sum_{q=q_1q_2q_3}\gamma^{(1)}_{q_1}\gamma^{(2)}_{q_2}\gamma^{(3)}_{q_3}.
\]
\end{enumerate}
\end{dfntn}
%
%
%
%
%%%%%%%%%%%%%%%%%%%%%%%%%%%%%%%%%%%%%%%%%%%%%%%%%%%%%%%%%%
%
%
%
%
With this definition, we are able to state our main result.
%
%
%
%
%%%%%%%%%%%%%%%%%%%%%%%%%%%%%%%%%%%%%%%%%%%%%%%%%%%%%%%%%%
%
%
%
%
\begin{thrm}\label{thrm:Factorable}
Let $a\in\mathbb{Z}$ and $A,\epsilon>0$. Let $\lambda_q$ be triply well factorable of level $Q\le x^{3/5-\epsilon}$. Then we have
\[
\sum_{\substack{q\le Q\\ (a,q)=1}}\lambda_q\Bigl(\pi(x;q,a)-\frac{\pi(x)}{\phi(q)}\Bigr)\ll_{a,A,\epsilon} \frac{x}{(\log{x})^A}.
\]
\end{thrm}
%
%
%
%
%%%%%%%%%%%%%%%%%%%%%%%%%%%%%%%%%%%%%%%%%%%%%%%%%%%%%%%%%%
%
%
%
%
The main point of this theorem is the quantitative improvement over Theorem \ref{nthrm:BFI1} allowing us to handle moduli as large as $x^{3/5-\epsilon}$ (instead of $x^{4/7-\epsilon}$). Theorem \ref{thrm:Factorable} has the disadvantage that it has a stronger requirement that the weights be triply well factorable rather than merely well-factorable, but we expect that Theorem \ref{thrm:Factorable} (or the ideas underlying it) will enable us to obtain quantitative improvements to several problems in analytic number theory where the best estimates currently rely on Theorem \ref{nthrm:BFI1}.

It appears that handling moduli of size $x^{3/5-\epsilon}$ is the limit of the current method. In particular, there appears to be no further benefit of imposing stronger constraints on the coefficients such as being `quadruply well factorable'.

As mentioned above, the main applications of such results come when using sieves. Standard sieve weights are not well-factorable (and so not triply well factorable), but Iwaniec \cite{IwaniecFactorable} showed that a slight variant of the upper bound $\beta$-sieve weights of level $D$ (which produces essentially identical results to the standard $\beta$-sieve weights) is a linear combination of sequences which are well-factorable of level $D$ provided $\beta\ge 1$. In particular, Theorem \ref{nthrm:BFI1} applies to the factorable variant of the upper bound sieve weights for the linear ($\beta=1$) sieve, for example. 

The factorable variant of the $\beta$-sieve weights of level $D$ are a linear combination of triply well factorable sequences of level $D$ provided $\beta\ge 2$, and so Theorem \ref{thrm:Factorable} automatically applies to these weights. Unfortunately it is the linear ($\beta=1$) sieve weights which are most important for many applications, and these are not triply well factorable of level $D$ (despite essentially being well-factorable). Despite this, the linear sieve weights have good factorization properties, and turns out that linear sieve weights of level $x^{7/12}$ are very close to being triply well factorable of level $x^{3/5}$. In particular, we have the following result.
%
%
%
%
%%%%%%%%%%%%%%%%%%%%%%%%%%%%%%%%%%%%%%%%%%%%%%%%%%%%%%%%%%
%
%
%
%
\begin{thrm}\label{thrm:Linear}
Let $a\in\mathbb{Z}$ and $A,\epsilon>0$. Let $\lambda^+_d$ be the well-factorable upper bound sieve weights for the linear sieve of level $D\le x^{7/12-\epsilon}$. Then we have
\[
\sum_{\substack{q\le x^{7/12-\epsilon}\\ (q,a)=1} }\lambda_q^+\Bigl(\pi(x;q,a)-\frac{\pi(x)}{\phi(q)}\Bigr)\ll_{a,A,\epsilon} \frac{x}{(\log{x})^A}.
\]
\end{thrm}
%
%
%
%
%%%%%%%%%%%%%%%%%%%%%%%%%%%%%%%%%%%%%%%%%%%%%%%%%%%%%%%%%%
%
%
%
%
This enables us to get good savings for the error term weighted by the linear sieve for larger moduli than was previously known. In particular, Theorem \ref{thrm:Linear} extends the range of moduli we are able to handle from the from the Bombieri-Friedlander-Iwaniec result \cite[Theorem 10]{BFI1} handling moduli of size $x^{4/7-\epsilon}$ to dealing with moduli of size $x^{7/12-\epsilon}$.

It is likely Theorem \ref{thrm:Linear} directly improves several results based on sieves. It doesn't directly improve upon estimates such as the upper bound for the number of twin primes, but we expect the underlying methods to give a suitable improvement for several such applications when combined with technique such as Chen's switching principle or Harman's sieve (see \cite{Harman,Chen,FouvryGrupp,FouvryGrupp2,Wu}). We intend to address this and related results in future work. Moreover, we expect that there are other upper bound sieves closely related to the linear sieve which are much closer to triply well factorable, and so we expect technical variants of Theorem \ref{thrm:Linear} adapted to these sieve weights to give additional improvements.
%
%
%
%%%%%%%%%%%%%%%%%%%%%%%%%%%%%%%%%%%%%%%%%%%%%%%%%%%%%%%%%%
%
%
%
\begin{rmk}
Drappeau \cite{Drappeau} proved equidistribution for smooth numbers in arithmetic progressions to moduli $x^{3/5-\epsilon}$. The advantage of flexible factorizations of smooth numbers allows one to use the most efficient estimates on convolutions (but one has to overcome additional difficulties in secondary main terms). Since our work essentially reduces to the original estimates of Bombieri--Friedlander--Iwaniec in these cases, it provides no benefit in this setting, but partially explains why we have the same limitiation of $x^{3/5-\epsilon}$.
\end{rmk}
%
%
%
%
%%%%%%%%%%%%%%%%%%%%%%%%%%%%%%%%%%%%%%%%%%%%%%%%%%%%%%%%%%
%
%
%
%
\section{Proof outline}
%
%
%
%
%%%%%%%%%%%%%%%%%%%%%%%%%%%%%%%%%%%%%%%%%%%%%%%%%%%%%%%%%%
%
%
%
%
The proof of Theorem \ref{thrm:Factorable} is a generalization of the method used to prove Theorem \ref{nthrm:BFI1}, and essentially includes the proof of Theorem \ref{nthrm:BFI1} as a special case. As with previous approaches, we use a combinatorial decomposition for the primes (Heath-Brown's identity) to reduce the problem to estimating bilinear quantities in arithmetic progressions. By Fourier expansion and several intermediate manipulations this reduces to estimating certain multidimensional exponential sums, which are ultimately bounded using the work of Deshouillers-Iwaniec \cite{DeshouillersIwaniec} coming from the spectral theory of automorphic forms via the Kuznetsov trace formula.

To obtain an improvement over the previous works we exploit the additional flexibility of factorizations of the moduli to benefit from the fact that now the weights can be factored into three pieces rather than two. This gives us enough room to balance the sums appearing from diagonal and off-diagonal terms perfectly in a wide range.

More specifically, let us recall the main ideas behind Theorem \ref{nthrm:BFI1}. A combinatorial decomposition leaves us to estimate for various ranges of $N,M,Q,R$
\[
\sum_{q\sim Q}\gamma_q \sum_{r\sim R}\lambda_r \sum_{m\sim M}\beta_m\sum_{n\sim N}\alpha_n\Bigl(\mathbf{1}_{n m\equiv a\Mod{qr}}-\frac{\mathbf{1}_{(n m,qr)=1}}{\phi(qr)}\Bigr),
\]
for essentially arbitrary 1-bounded sequences $\gamma_q,\lambda_r,\beta_m,\alpha_n$. Applying Cauchy-Schwarz in the $m$, $q$ variables and then Fourier expanding the $m$-summation and using Bezout's identity reduces this to bounding something like
\[
\sum_{q\sim Q}\sum_{\substack{n_1,n_2\sim N\\ n_1\equiv n_2\Mod{q}}}\alpha_{n_1}\overline{\alpha_{n_2}}\sum_{r_1,r_2\sim R}\lambda_{r_1}\overline{\lambda_{r_2}}\sum_{h\sim H}e\Bigl(\frac{ah \overline{n_2 q r_1}(n_1-n_2)}{n_1r_2}\Bigr),
\]
where $H\approx NQR^2/x$. Writing $n_1-n_2=qf$ and switching the $q$-summation to an $f$-summation, then applying Cauchy-Schwarz in the $n_1,n_2,f,r_2$ variables leaves us to bound
\[
\sum_{f\sim N/Q}\sum_{\substack{n_1,n_2\sim N\\ n_1\equiv n_2\Mod{f}}}\sum_{r_2\sim R}\Bigl|\sum_{r_1\sim R}\gamma_{r_1}\sum_{h\sim H}e\Bigl(\frac{a h f \overline{r_1 n_2}}{n_1r_2}\Bigr)\Bigr|^2.
\]
Bombieri-Friedlander-Iwaniec then drop the congruence condition on $n_1,n_2$, combine $n_1,r_2$ into a new variable $c$ and then estimate the resulting exponential sums via the bounds of Deshouillers-Iwaniec. This involves applying the Kuznetsov trace formula for the congruence subgroup $\Gamma_0(r_1r_1')$. This is a large level (of size $R^2$), which means that the resulting bounds deteriorate rapidly with $R$. We make use of the fact that if the moduli factorize suitably, then we can reduce this level at the cost of worsening the diagonal terms slightly. 

In particular, if the above $\lambda_r$ coefficients were of the form $\kappa_s\star\nu_t$, then instead we could apply the final Cauchy-Schwarz in $f$, $n_1$, $n_2$, $r_2$ and $s_1$, leaving us instead to bound
\[
\sum_{f\sim N/Q}\sum_{\substack{n_1,n_2\sim N\\ n_1\equiv n_2\Mod{f}}}\sum_{r_2\sim R}\sum_{s_1\sim S}\Bigl|\sum_{t_1\sim T}\nu_{t_1}\sum_{h\sim H}e\Bigl(\frac{a h f \overline{s_1 t_1 n_2}}{n_1r_2}\Bigr)\Bigr|^2.
\]
(Here $ST\approx R$). Here we have increased the diagonal contribution by a factor $S$, but now the level of the relevant congruence subgroup has dropped from $R^2$ to $T^2$. By dropping the congruence condition, combining $c=n_1r_2$ and $d=n_2s_1$, we can then apply the Deshouillers-Iwaniec estimates in a more efficient manner, giving an additional saving over the previous approach in all the important regimes. This ultimately allows us to handle moduli as large as $x^{3/5-\epsilon}$ in Theorem \ref{thrm:Factorable}. On its own this approach doesn't quite cover all relevant ranges for $N$, but combining it with known estimates for the divisor function in arithmetic progressions (based on the Weil bound) allows us to cover the remaining ranges.

We view the main interest of Theorem \ref{thrm:Factorable} and Theorem \ref{thrm:Linear} as their applicability to sieve problems. It is therefore unfortunate that Theorem \ref{thrm:Factorable} doesn't apply directly to the (well-factorable variant of the) linear sieve weights. To overcome this limitation, it is therefore necessary for us to exploit the fact that our main technical result on convolutions (Proposition \ref{prpstn:MainProp}) actually gives a stronger estimate than what is captured by Theorem \ref{thrm:Factorable}. Moreover, it is necessary to study the precise construction of the linear sieve weights to show that they enjoy good factorization properties. Indeed, we recall the support set for the upper bound linear sieve weights of level $D$ is 
\[
\mathcal{D}^+(D)=\Bigl\{p_1\cdots p_r:\, p_1\ge p_2\ge \dots \ge p_r,\,\,p_1\cdots p_{2j}p_{2j+1}^3\le D\text{ for $0\le j<r/2$}\Bigr\}.
\]
If $p_1\cdots p_r$ is close to $D$, we must have the most of the $p_i$'s are very small, and so the weights are supported on very well factorable numbers. It is only really the largest few prime factors $p_i$ which obstruct finding factors in given ranges, and so by explicitly handling them we can exploit this structure much more fully.

Proposition \ref{prpstn:Factorization} is a technical combinatorial proposition showing that the the linear sieve weights enjoy rather stronger factorization properties that simply what is captured through being well-factorable. Although these are not sufficient for triple well-factorability, they are sufficient for our more technical conditions coming from Proposition \ref{prpstn:MainProp}. This ultimately leads to Theorem \ref{thrm:Linear}.
%
%
%
%
%%%%%%%%%%%%%%%%%%%%%%%%%%%%%%%%%%%%%%%%%%%%%%%%%%%%%%%%%%
%
%
%
%
\section{Acknowledgements}
%
%
%
%
%%%%%%%%%%%%%%%%%%%%%%%%%%%%%%%%%%%%%%%%%%%%%%%%%%%%%%%%%%
%
%
%
%
I would like to thank John Friedlander, Ben Green, Henryk Iwaniec and Kyle Pratt for useful discussions and suggestions. JM is supported by a Royal Society Wolfson Merit Award, and this project has received funding from the European Research Council (ERC) under the European Union’s Horizon 2020 research and innovation programme (grant agreement No 851318).
%
%
%
%
%%%%%%%%%%%%%%%%%%%%%%%%%%%%%%%%%%%%%%%%%%%%%%%%%%%%%%%%%%
%
%
%
%
\section{Notation}
%
%
%
%
%%%%%%%%%%%%%%%%%%%%%%%%%%%%%%%%%%%%%%%%%%%%%%%%%%%%%%%%%%
%
%
%
%
We will use the Vinogradov $\ll$ and $\gg$ asymptotic notation, and the big oh $O(\cdot)$ and $o(\cdot)$ asymptotic notation. $f\asymp g$ will denote the conditions $f\ll g$ and $g\ll f$ both hold. Dependence on a parameter will be denoted by a subscript. We will view $a$ (the residue class $\Mod{q}$) as a fixed positive integer throughout the paper, and any constants implied by asymptotic notation will be allowed to depend on $a$ from this point onwards. Similarly, throughout the paper, we will let $\epsilon$ be a single fixed small real number; $\epsilon=10^{-100}$ would probably suffice. Any bounds in our asymptotic notation will also be allowed to depend on $\epsilon$.

The letter $p$ will always be reserved to denote a prime number. We use $\phi$ to denote the Euler totient function, $e(x):=e^{2\pi i x}$ the complex exponential, $\tau_k(n)$ the $k$-fold divisor function, $\mu(n)$ the M\"obius function. We let $P^-(n)$, $P^+(n)$ denote the smallest and largest prime factors of $n$ respectively, and $\hat{f}$ denote the Fourier transform of $f$ over $\mathbb{R}$ - i.e. $\hat{f}(\xi)=\int_{-\infty}^{\infty}f(t)e(-\xi t)dt$. We use $\mathbf{1}$ to denote the indicator function of a statement. For example,
\[
\mathbf{1}_{n\equiv a\Mod{q}}=\begin{cases}1,\qquad &\text{if }n\equiv a\Mod{q},\\
0,&\text{otherwise}.
\end{cases}
\]
For $(n,q)=1$, we will use $\overline{n}$ to denote the inverse of the integer $n$ modulo $q$; the modulus will be clear from the context. For example, we may write $e(a\overline{n}/q)$ - here $\overline{n}$ is interpreted as the integer $m\in \{0,\dots,q-1\}$ such that $m n\equiv 1\Mod{q}$. Occasionally we will also use $\overline{\lambda}$ to denote complex conjugation; the distinction of the usage should be clear from the context.  For a complex sequence $\alpha_{n_1,\dots,n_k}$, $\|\alpha\|_2$ will denote the $\ell^2$ norm $\|\alpha\|_2=(\sum_{n_1,\dots,n_k}|\alpha_{n_1,\dots,n_k}|^2)^{1/2}$.

Summations assumed to be over all positive integers unless noted otherwise. We use the notation $n\sim N$ to denote the conditions $N<n\le 2N$.

We will let $z_0:=x^{1/(\log\log{x})^3}$ and $y_0:=x^{1/\log\log{x}}$ two parameters depending on $x$, which we will think of as a large quantity. We will let $\psi_0:\mathbb{R}\rightarrow\mathbb{R}$ denote a fixed smooth function supported on $[1/2,5/2]$ which is identically equal to $1$ on the interval $[1,2]$ and satisfies the derivative bounds $\|\psi_0^{(j)}\|_\infty\ll (4^j j!)^2$ for all $j\ge 0$. (See \cite[Page 368, Corollary]{BFI2} for the construction of such a function.)

We will repeatedly make use of the following condition.
\begin{dfntn}[Siegel-Walfisz condition]
We say that a complex sequence $\alpha_n$ satisfies the \textbf{Siegel-Walfisz condition} if for every $d\ge 1$, $q\ge 1$ and $(a,q)=1$ and every $A>1$ we have
\begin{equation}
\Bigl|\sum_{\substack{n\sim N\\ n\equiv a\Mod{q}\\ (n,d)=1}}\alpha_n-\frac{1}{\phi(q)}\sum_{\substack{n\sim N\\ (n,d q)=1}}\alpha_n\Bigr|\ll_A \frac{N\tau(d)^{O(1)}}{(\log{N})^A}.
\label{eq:SiegelWalfisz}
\end{equation}
\end{dfntn}
We note that $\alpha_n$ satisfies the Siegel-Walfisz condition if $\alpha_n=1$ or if $\alpha_n=\mu(n)$.
%
%
%
%
%%%%%%%%%%%%%%%%%%%%%%%%%%%%%%%%%%%%%%%%%%%%%%%%%%%%%%%%%%
%
%
%
%
\section{Proof of Theorem \ref{thrm:Factorable}}\label{sec:Factorable}
%
%
%
%
%%%%%%%%%%%%%%%%%%%%%%%%%%%%%%%%%%%%%%%%%%%%%%%%%%%%%%%%%%
%
%
%
%
In this section we establish Theorem \ref{thrm:Factorable} assuming two propositions, namely Proposition \ref{prpstn:WellFactorable} and Proposition \ref{prpstn:DoubleDivisor}, given below.
%
%
%
%
%%%%%%%%%%%%%%%%%%%%%%%%%%%%%%%%%%%%%%%%%%%%%%%%%%%%%%%%%%
%
%
%
%
\begin{prpstn}[Well-factorable Type II estimate]\label{prpstn:WellFactorable}
Let $\lambda_q$ be triply well factorable of level $Q\le x^{3/5-10\epsilon}$, let $NM\asymp x$ with
\[
x^\epsilon\le N\le x^{2/5}.
\]
Let $\alpha_n,\beta_m$ be complex sequences such that $|\alpha_n|,|\beta_n|\le \tau(n)^{B_0}$ and $\alpha_n$ satisfies the Siegel-Walfisz condition \eqref{eq:SiegelWalfisz} and is supported on $P^-(n)\ge z_0$. Then we have that for every choice of $A>0$ and every interval $\mathcal{I}\subseteq[x,2x]$
\[
\sum_{q\le Q}\lambda_q\sum_{n\sim N}\alpha_n\sum_{\substack{m\sim M\\ mn\in\mathcal{I}}}\beta_m\Bigl(\mathbf{1}_{nm\equiv a\Mod{q}}-\frac{\mathbf{1}_{(nm,q)=1}}{\phi(q)}\Bigr)\ll_{A,B_0}\frac{x}{(\log{x})^A}.
\]
\end{prpstn}
%
%
%
%
%%%%%%%%%%%%%%%%%%%%%%%%%%%%%%%%%%%%%%%%%%%%%%%%%%%%%%%%%%
%
%
%
%
Proposition \ref{prpstn:WellFactorable} is our key new ingredient behind the proof, and will be established in Section \ref{sec:WellFactorable}.
%
%
%
%
%%%%%%%%%%%%%%%%%%%%%%%%%%%%%%%%%%%%%%%%%%%%%%%%%%%%%%%%%%
%
%
%
%
\begin{prpstn}[Divisor function in arithmetic progressions]\label{prpstn:DoubleDivisor}
Let $N_1,N_2\ge x^{3\epsilon}$ and $N_1N_2M\asymp x$ and
\begin{align*}
Q&\le \Bigl(\frac{x}{M}\Bigr)^{2/3-3\epsilon}.
\end{align*}
Let $\mathcal{I}\subset[x,2x]$ be an interval, and let $\alpha_m$ a complex sequence with $|\alpha_m|\le \tau(m)^{B_0}$. Then we have that for every $A>0$
\[
\sum_{q\sim Q}\Bigl|\sum_{\substack{n_1\sim N_1\\ P^-(n)\ge z_0}}\sum_{\substack{n_2\sim N_2\\ P^-(n)\ge z_0}}\sum_{\substack{m\sim M\\ m n_1n_2\in\mathcal{I} }}\alpha_m\Bigl(\mathbf{1}_{m n_1 n_2\equiv a\Mod{q}}-\frac{\mathbf{1}_{(m n_1 n_2,q)=1}}{\phi(q)}\Bigr)\Bigr|\ll_{A,B_0} \frac{x}{(\log{x})^A}.
\]
Moreover, the same result holds when the summand is multiplied by $\log{n_1}$.
\end{prpstn}
%
%
%
%
%%%%%%%%%%%%%%%%%%%%%%%%%%%%%%%%%%%%%%%%%%%%%%%%%%%%%%%%%%
%
%
%
%
Proposition \ref{prpstn:DoubleDivisor} is essentially a known result (due to independent unpublished work of Selberg and Hooley, but following quickly from the Weil bound for Kloosterman sums), but for concreteness we give a proof in Section \ref{sec:DoubleDivisor}. 

Finally, we require a suitable combinatorial decomposition of the primes.
%
%
%
%
%%%%%%%%%%%%%%%%%%%%%%%%%%%%%%%%%%%%%%%%%%%%%%%%%%%%%%%%%%
%
%
%
%
\begin{lmm}[Heath-Brown identity]\label{lmm:HeathBrown}
Let $k\ge 1$ and $n\le 2x$. Then we have
\[
\Lambda(n)=\sum_{j=1}^k (-1)^j \binom{k}{j}\sum_{\substack{n=n_1\cdots n_{j}m_1\cdots m_j\\ m_1,\dots,m_j\le 2x^{1/k}}}\mu(m_1)\cdots \mu(m_j)\log{n_{1}}.
\]
\end{lmm}
\begin{proof}
See \cite{HBVaughan}.
\end{proof}
%
%
%
%
%%%%%%%%%%%%%%%%%%%%%%%%%%%%%%%%%%%%%%%%%%%%%%%%%%%%%%%%%%
%
%
%
%
\begin{lmm}[Consequence of the fundamental lemma of the sieve]\label{lmm:Fundamental}
Let $q,t,x\ge 2$ satisfy $q x^\epsilon \le t$ and let $(b,q)=1$. Recall $z_0=x^{1/(\log\log{x})^3}$. Then we have
\[
\sum_{\substack{n\le t\\ n\equiv b\Mod{q}\\ P^-(n)\ge z_0}}1=\frac{1}{\phi(q)}\sum_{\substack{n\le t\\ P^-(n)\ge z_0}}1+O_A\Bigl(\frac{t}{q(\log{x})^A}\Bigr).
\]
\end{lmm}
\begin{proof}
This is an immediate consequence of the fundamental lemma of sieve methods - see, for example, \cite[Theorem 6.12]{Opera}.
\end{proof}
%
%
%
%
%%%%%%%%%%%%%%%%%%%%%%%%%%%%%%%%%%%%%%%%%%%%%%%%%%%%%%%%%%
%
%
%
%
\begin{proof}[Proof of Theorem \ref{thrm:Factorable} assuming Proposition \ref{prpstn:WellFactorable} and \ref{prpstn:DoubleDivisor}]
By partial summation (noting that prime powers contribute negligibly and retaining the conditon $P^-(n)\ge z_0$), it suffices to show that for all $t\in [x,2x]$
\[
\sum_{q\le x^{3/5-\epsilon} }\lambda_q\sum_{\substack{x\le n\le t\\ P^-(n)\ge z_0}}\Lambda(n)\Bigl(\mathbf{1}_{n\equiv a\Mod{q}}-\frac{\mathbf{1}_{(n,q)=1}}{\phi(q)}\Bigr)\ll_A \frac{x}{(\log{x})^A}.
\]
We now apply Lemma \ref{lmm:HeathBrown} with $k=3$ to expand $\Lambda(n)$ into various subsums, and put each variable into one of $O(\log^6{x})$ dyadic intervals. Thus it suffices to show that for all choices of $N_1,N_2,N_3,M_1,M_2,M_3$ with $M_1M_2M_3N_1N_2N_3\asymp x$ and $M_i\le x^{1/3}$ we have
\begin{align*}
\sum_{q\le x^{3/5-\epsilon} }\lambda_q\sum_{\substack{m_1,m_2,m_3,n_1,n_2,n_3\\ n_i\sim N_i\,\forall i\\ m_i\sim M_i\,\forall i\\ x\le n \le t\\ P^-(n_i),P^-(m_i)\ge z_0\,\forall i}}\mu(m_1)\mu(m_2)\mu(m_3)(\log{n_1})\Bigl(\mathbf{1}_{n\equiv a\Mod{q}}-\frac{\mathbf{1}_{(n,q)=1}}{\phi(q)}\Bigr)\\
\ll_A \frac{x}{(\log{x})^{A+6}},
\end{align*}
where we have written $n=n_1n_2n_3m_1m_2m_3$ in the expression above for convenience. 

By grouping all but one variable together, Proposition \ref{prpstn:WellFactorable} gives this if any of the $N_i$ or $M_i$ lie in the interval $[x^\epsilon,x^{2/5}]$, and so we may assume all are either smaller than $x^\epsilon$ or larger than $x^{2/5}$. Since $M_i\le x^{1/3}\le x^{2/5}$, we may assume that $M_1,M_2,M_3\le x^\epsilon$. There can be at most two of the $N_i$'s which are larger than $x^{2/5}$ since $M_1M_2M_3N_1N_2N_3\asymp x$. 

If only one of the $N_i$'s are greater than $x^{2/5}$ then they must be of size $\gg x^{1-5\epsilon}>x^\epsilon q$, and so the result is trivial by summing over this variable first and using Lemma \ref{lmm:Fundamental}.

If two of the $N_i$'s are larger than $x^{2/5}$ and all the other variables are less than $x^\epsilon$, then the result follows immediately from Proposition \ref{prpstn:DoubleDivisor}. This gives the result.
\end{proof}
%
%
%
%
%%%%%%%%%%%%%%%%%%%%%%%%%%%%%%%%%%%%%%%%%%%%%%%%%%%%%%%%%%
%
%
%
%
To complete the proof of Theorem \ref{thrm:Factorable}, we are left to establish Propositions \ref{prpstn:WellFactorable} and \ref{prpstn:DoubleDivisor}, which we will ultimately do in Sections \ref{sec:WellFactorable} and \ref{sec:DoubleDivisor} respectively.
%
%
%
%
%%%%%%%%%%%%%%%%%%%%%%%%%%%%%%%%%%%%%%%%%%%%%%%%%%%%%%%%%%
%
%
%
%
\section{Preparatory lemmas}\label{sec:Lemmas}
%
%
%
%
%%%%%%%%%%%%%%%%%%%%%%%%%%%%%%%%%%%%%%%%%%%%%%%%%%%%%%%%%%
%
%
%
%
\begin{lmm}[Divisor function bounds]\label{lmm:Divisor}
Let $|b|< x-y$ and $y\ge q x^\epsilon$. Then we have
\[
\sum_{\substack{x-y\le n\le x\\ n\equiv a\Mod{q}}}\tau(n)^C\tau(n-b)^C\ll \frac{y}{q} (\tau(q)\log{x})^{O_{C}(1)}.
\]
\end{lmm}
\begin{proof}
This follows from Shiu's Theorem \cite{Shiu}, and is given in \cite[Lemma 7.7]{May1}.
\end{proof}
%
%
%
%
%%%%%%%%%%%%%%%%%%%%%%%%%%%%%%%%%%%%%%%%%%%%%%%%%%%%%%%%%%
%
%
%
%
\begin{lmm}[Separation of variables from inequalities]\label{lmm:Separation}
Let $Q_1Q_2\le x^{1-\epsilon}$. Let $N_1,\dots, N_r\ge z_0$ satisfy $N_1\cdots N_r\asymp x$. Let $\alpha_{n_1,\dots,n_r}$ be a complex sequence with $|\alpha_{n_1,\dots,n_r}|\le (\tau(n_1)\cdots \tau(n_r))^{B_0}$. Then, for any choice of $A>0$ there is a constant $C=C(A,B_0,r)$ and intervals $\mathcal{I}_1,\dots,\mathcal{I}_r$ with $\mathcal{I}_j\subseteq [P_j,2P_j]$ of length $\le P_j(\log{x})^{-C}$ such that
\begin{align*}
\sum_{q_1\sim Q_1}\sum_{\substack{q_2\sim Q_2\\ (q_1q_2,a)=1}}&\Bigl|\,\sideset{}{^*}\sum_{\substack{n_1,\dots,n_r\\ n_i\sim N_i\forall i}}\alpha_{n_1,\dots,n_r}S_{n_1\cdots n_r}\Bigr|\\
&\ll_r \frac{x}{(\log{x})^A}+(\log{x})^{r C}\sum_{q_1\sim Q_1}\sum_{\substack{q_2\sim Q_2\\ (q_1q_2,a)=1}}\Bigl|\sum_{\substack{n_1,\dots,n_r\\ n_i\in \mathcal{I}_i\forall i}}\alpha_{n_1,\dots,n_r}S_{n_1\cdots n_r}\Bigr|.
\end{align*}
Here $\sum^*$ means that the summation is restricted to $O(1)$ inequalities of the form $n_1^{\alpha_1}\cdots n_r^{\alpha_r}\le B$ for some constants $\alpha_1,\dots \alpha_r$ and some quantity $B$.  The implied constant may depend on all such exponents $\alpha_i$, but none of the quantities $B$.
\end{lmm}
\begin{proof}
This is \cite[Lemma 7.10]{May1}.
\end{proof}
%
%
%
%
%%%%%%%%%%%%%%%%%%%%%%%%%%%%%%%%%%%%%%%%%%%%%%%%%%%%%%%%%%
%
%
%
%
\begin{lmm}[Poisson Summation]\label{lmm:Completion}
Let $C>0$ and $f:\mathbb{R}\rightarrow\mathbb{R}$ be a smooth function which is supported on $[-10,10]$ and satisfies $\|f^{(j)}\|_\infty\ll_j (\log{x})^{j C}$ for all $j\ge 0$, and let $M,q\le x$. Then we have
\[
\sum_{m\equiv a\Mod{q}} f\Bigl(\frac{m}{M}\Bigr)=\frac{M}{q}\hat{f}(0)+\frac{M}{q}\sum_{1\le |h|\le H}\hat{f}\Bigl(\frac{h M}{q}\Bigr)e\Bigl(\frac{ah}{q}\Bigr)+O_C(x^{-100}),
\]
for any choice of $H>q x^\epsilon/M$.
\end{lmm}
\begin{proof}
This follows from \cite[Lemma 12.4]{May1}.
\end{proof}
%
%
%
%
%%%%%%%%%%%%%%%%%%%%%%%%%%%%%%%%%%%%%%%%%%%%%%%%%%%%%%%%%%
%
%
%
%
\begin{lmm}[Summation with coprimality constraint]\label{lmm:TrivialCompletion}
Let $C>0$ and $f:\mathbb{R}\rightarrow\mathbb{R}$ be a smooth function which is supported on $[-10,10]$ and satisfies $\|f^{(j)}\|_\infty\ll_j (\log{x})^{j C}$ for all $j\ge 0$. Then we have
\[
\sum_{(m,q)=1}f\Bigl(\frac{m}{M}\Bigr)=\frac{\phi(q)}{q}M+O(\tau(q)(\log{x})^{2C}).
\]
\end{lmm}
\begin{proof}
This is \cite[Lemma 12.6]{May1}.
\end{proof}
%
%
%
%
%%%%%%%%%%%%%%%%%%%%%%%%%%%%%%%%%%%%%%%%%%%%%%%%%%%%%%%%%%
%
%
%
%
\begin{lmm}\label{lmm:SiegelWalfiszMaintain}
Let $C,B>0$ be constants and let $\alpha_n$ be a sequence satisfing the Siegel-Walfisz condition \eqref{eq:SiegelWalfisz}, supported on $n\le 2x$ with $P^-(n)\ge z_0=x^{1/(\log\log{x})^3}$ and satisfying $|\alpha_n|\le \tau(n)^B$. Then $\mathbf{1}_{\tau(n)\le (\log{x})^C}\alpha_n$ also satisfies the Siegel-Walfisz condition.
\end{lmm}
\begin{proof}
This is \cite[Lemma 12.7]{May1}.
\end{proof}
%
%
%
%
%%%%%%%%%%%%%%%%%%%%%%%%%%%%%%%%%%%%%%%%%%%%%%%%%%%%%%%%%%
%
%
%
%
\begin{lmm}[Most moduli have small square-full part]\label{lmm:Squarefree}
Let $\gamma_b,c_q$ be complex sequences satisfying $|\gamma_b|, |c_b|\le \tau(b)^{B_0}$ and recall $z_0:=x^{1/(\log\log{x})^3}$. Let $sq(n)$ denote the square-full part of $n$. (i.e. $sq(n)=\prod_{p:p^2|n}p^{\nu_p(n)}$). Then for every $A>0$ we have that
\[
\sum_{\substack{q\sim Q\\ sq(q)\ge z_0}}c_q\sum_{b\le  B}\gamma_b\Bigl(\mathbf{1}_{b\equiv a\Mod{q}}-\frac{\mathbf{1}_{(b,q)=1}}{\phi(q)}\Bigr)\ll_{A,B_0} \frac{x}{(\log{x})^A}.
\]
\end{lmm}
\begin{proof}
This is \cite[Lemma 12.9]{May1}.
\end{proof}
%
%
%
%
%%%%%%%%%%%%%%%%%%%%%%%%%%%%%%%%%%%%%%%%%%%%%%%%%%%%%%%%%%
%
%
%
%
\begin{lmm}[Most moduli have small $z_0$-smooth part]\label{lmm:RoughModuli}
Let $Q<x^{1-\epsilon}$. Let $\gamma_b,c_q$ be complex sequences with $|\gamma_b|,|c_b|\le \tau(n)^{B_0}$ and recall $z_0:=x^{1/(\log\log{x})^3}$ and $y_0:=x^{1/\log\log{x}}$. Let $sm(n;z)$ denote the $z$-smooth part of $n$. (i.e. $sm(n;z)=\prod_{p\le z}p^{\nu_p(n)}$). Then for every $A>0$ we have that
\[
\sum_{\substack{q\sim Q\\ sm(q;z_0)\ge y_0}}c_q\sum_{b\le  x}\gamma_b\Bigl(\mathbf{1}_{b\equiv a\Mod{q}}-\frac{\mathbf{1}_{(b,q)=1}}{\phi(q)}\Bigr)\ll_{A,B_0} \frac{x}{(\log{x})^A}.
\]
\end{lmm}
\begin{proof}
This is \cite[Lemma 12.10]{May1}.
\end{proof}
%
%
%
%
%%%%%%%%%%%%%%%%%%%%%%%%%%%%%%%%%%%%%%%%%%%%%%%%%%%%%%%%%%
%
%
%
%
\begin{prpstn}[Reduction to exponential sums]\label{prpstn:GeneralDispersion}
Let $\alpha_n,\beta_m,\gamma_{q,d},\lambda_{q,d,r}$ be complex sequences with $|\alpha_n|,|\beta_n|\le \tau(n)^{B_0}$ and $|\gamma_{q,d}|\le \tau(q d)^{B_0}$ and $|\lambda_{q,d,r}|\le \tau(q d r)^{B_0}$. Let $\alpha_n$ and $\lambda_{q,d,r}$ be supported on integers with $P^-(n)\ge z_0$ and $P^-(r)\ge z_0$, and let $\alpha_n$ satisfy the Siegel-Walfisz condition \eqref{eq:SiegelWalfisz}. Let
\[
\mathscr{S}:=\sum_{\substack{d\sim D\\ (d,a)=1}}\sum_{\substack{q\sim Q\\ (q,a)=1}}\sum_{\substack{r\sim R\\ (r,a)=1}}\lambda_{q,d,r}\gamma_{q,d}\sum_{m\sim M}\beta_m\sum_{n\sim N}\alpha_n\Bigl(\mathbf{1}_{m n\equiv a\Mod{q r d}}-\frac{\mathbf{1}_{(m n,q r d)=1}}{\phi(q r d)}\Bigr).
\]
Let $A>0$ and $C=C(A,B_0)$ be sufficiently large in terms of $A,B_0$, and let $N,M$ satisfy
\[
N>Q D (\log{x})^{C},\qquad M>(\log{x})^C.
\]
Then we have
\[
|\mathscr{S}|\ll_{A,B_0} \frac{x}{(\log{x})^A}+M D^{1/2}Q^{1/2}(\log{x})^{O_{B_0}(1)}\Bigl(|\mathscr{E}_1|^{1/2}+|\mathscr{E}_2|^{1/2}\Bigr),
\]
where
\begin{align*}
\mathscr{E}_{1}&:=\sum_{\substack{q\\ (q,a)=1}}\sum_{\substack{d\sim D\\ (d,a)=1}}\sum_{\substack{r_1,r_2\sim R\\ (r_1r_2,a)=1}}\psi_0\Bigl(\frac{q}{Q}\Bigr)\frac{\lambda_{q,d,r_1}\overline{\lambda_{q,d,r_2}} }{\phi(q d r_2)q d r_1}\sum_{\substack{n_1,n_2\sim N\\ (n_1,q d r_1)=1\\(n_2,q d  r_2)=1}}\alpha_{n_1}\overline{\alpha_{n_2}}\\
&\qquad \times\sum_{1\le |h|\le H_1}\hat{\psi}_0\Bigl(\frac{h M}{q d r_1}\Bigr)e\Bigl( \frac{a h \overline{ n_1}}{q d r_1}\Bigr),\\
\mathscr{E}_2&:=\sum_{\substack{q\\ (q,a)=1}}\psi_0\Bigl(\frac{q}{Q}\Bigr)\sum_{\substack{d\sim D\\ (d,a)=1}}\sum_{\substack{r_1,r_2\sim R\\ (r_1,a r_2)=1\\ (r_2,a q d r_1)=1}}\frac{\lambda_{q,d,r_1}\overline{\lambda_{q,d,r_2}}}{q d r_1 r_2}\sum_{\substack{n_1,n_2\sim N\\ n_1\equiv n_2\Mod{q d}\\ (n_1,n_2 q d r_1)=1\\(n_2,n_1 q d r_2)=1\\ |n_1-n_2|\ge N/(\log{x})^C}}\alpha_{n_1}\overline{\alpha_{n_2}}\\
&\qquad \times\sum_{1\le |h|\le H_2}\hat{\psi}_0\Bigl(\frac{h M}{q d r_1 r_2}\Bigr)e\Bigl(\frac{ah\overline{n_1r_2}}{q d r_1}+\frac{ah\overline{n_2 q d r_1}}{r_2}\Bigr),\\
H_1&:=\frac{Q D R}{M}\log^5{x},\\
H_2&:=\frac{Q D R^2}{M}\log^5{x}.
\end{align*}
\end{prpstn}
\begin{proof}
This is \cite[Proposition 13.4]{May1} with $E=1$.
\end{proof}
%
%
%
%
%%%%%%%%%%%%%%%%%%%%%%%%%%%%%%%%%%%%%%%%%%%%%%%%%%%%%%%%%%
%
%
%
%
\begin{lmm}[Simplification of exponential sum]\label{lmm:Simplification}
Let $N,M,Q,R \le x$ with $NM\asymp x$ and 
\begin{align}
Q R&<x^{2/3},\label{eq:CrudeSize}\\
Q R^2&< M x^{1-2\epsilon}.\label{eq:CrudeSize2}
\end{align}
Let $\lambda_{q,r}$ and $\alpha_n$ be complex sequences supported on $P^-(n),P^-(r)\ge z_0$ with $|\lambda_{q,r}|\le \tau(qr)^{B_0}$ and $|\alpha_n|\le \tau(n)^{B_0}$. Let $H:=\frac{Q R^2}{M}\log^5{x}$ and let
\begin{align*}
\mathscr{E}&:=\sum_{\substack{(q,a)=1}}\psi_0\Bigl(\frac{q}{Q}\Bigr)\sum_{\substack{r_1,r_2\sim R\\ (r_1,a r_2)=1\\ (r_2,a q r_2)=1}}\frac{\lambda_{q,r_1}\overline{\lambda_{q,r_2}}}{q r_1 r_2}\sum_{\substack{n_1,n_2\sim N\\ n_1\equiv n_2\Mod{q}\\ (n_1,n_2qr_1)=1\\(n_2,n_1qr_2)=1\\ |n_1-n_2|\ge N/(\log{x})^C}}\alpha_{n_1}\overline{\alpha_{n_2}}\\
&\qquad\qquad \times\sum_{1\le |h|\le H}\hat{\psi}_0\Bigl(\frac{h M}{q r_1 r_2}\Bigr)e\Bigl(\frac{ah\overline{n_1 r_2}}{q r_1}+\frac{ah\overline{n_2 q r_1}}{r_2}\Bigr).
\end{align*}
Then we have (uniformly in $C$)
\[
\mathscr{E}\ll_{B_0}\exp((\log\log{x})^5)\sup_{\substack{H'\le H\\ Q'\le 2Q\\ R_1,R_2\le 2R}}|\mathscr{E}'|+\frac{N^2}{Qx^\epsilon},
\]
where
\[
\mathscr{E}'=\sum_{\substack{Q\le q\le Q'\\ (q,a)=1}}\sum_{\substack{R\le r_1\le  R_1\\ R\le r_2\le R_2\\ (r_1a r_2)=1\\ (r_2,a q r_1)=1}}\frac{\lambda_{q,r_1}\overline{\lambda_{q,r_2}}}{q r_1 r_2}\sum_{\substack{n_1,n_2\sim N\\ n_1\equiv n_2\Mod{q}\\ (n_1,qr_1n_2)=1\\ (n_2,qr_2n_1)=1\\ (n_1r_2,n_2)\in\mathcal{N}\\ |n_1-n_2|\ge N/(\log{x})^C}}\alpha_{n_1}\overline{\alpha_{n_2}}\sum_{1\le |h| \le H'} e\Bigl(\frac{ ah\overline{n_2 q r_1}(n_1-n_2)}{n_1 r_2}\Bigr),
\]
and $\mathcal{N}$ is a set with the property that if $(a,b)\in\mathcal{N}$ and $(a',b')\in\mathcal{N}$ then we have $\gcd(a,b')=\gcd(a',b)=1$.
\end{lmm}
\begin{proof}
This is \cite[Lemma 13.5]{May1}.
\end{proof}
%
%
%
%
%%%%%%%%%%%%%%%%%%%%%%%%%%%%%%%%%%%%%%%%%%%%%%%%%%%%%%%%%%
%
%
%
%
\begin{lmm}[Second exponential sum estimate]\label{lmm:BFI2}
Let
\begin{align}
D R N^{3/2}&< x^{1-2\epsilon},\\
Q D R&< x^{1-2\epsilon}.
\end{align}
Let $\alpha_n$, $\lambda_{d,r}$ be complex sequences with $|\lambda_{d,r}|,|\alpha_n|\le x^{o(1)}$. Let $H_1:=N Q D R(\log{x})^5/x$ and let 
\[
\widetilde{\mathscr{B}}:=\sum_{\substack{q\\ (q,a)=1}}\sum_{\substack{d\sim D\\ (d,a)=1}}\sum_{\substack{r_1,r_2\sim R\\ (r_1r_2,a)=1}}\psi_0\Bigl(\frac{q}{Q}\Bigr)\frac{\lambda_{d,r_1}\overline{\lambda_{d,r_2}} }{\phi(q d r_2)q d r_1}\sum_{\substack{n_1,n_2\sim N\\ (n_1,q d r_1)=1\\(n_2,q d  r_2)=1}}\alpha_{n_1}\overline{\alpha_{n_2}}\sum_{1\le |h|\le H_1}\hat{\psi}_0\Bigl(\frac{h M}{q d r_1}\Bigr)e\Bigl( \frac{a h \overline{ n_1}}{q d r_1}\Bigr)
\]
Then we have
\[
\widetilde{\mathscr{B}}\ll\frac{N^2}{Q D x^\epsilon}.
\]
\end{lmm}
\begin{proof}
This follows from the same argument used to prove \cite[Lemma 17.3]{May1}.
\end{proof}
%
%
%
%
%%%%%%%%%%%%%%%%%%%%%%%%%%%%%%%%%%%%%%%%%%%%%%%%%%%%%%%%%%
%
%
%
%
\begin{lmm}[Reduction to smoothed sums]\label{lmm:SmoothReduction}
Let $N\ge x^\epsilon$ and $z\le z_0$ and let $\alpha_m$, $c_q$ be 1-bounded complex sequences.

Imagine that for every choice of $N',D,A,C>0$ with $N' D\asymp N$ and $D\le y_0$, and every smooth function $f$ supported on $[1/2,5/2]$ satisfying $f^{(j)}\ll_j (\log{x})^{C j}$, and for every $1$-bounded complex sequence $\beta_d$ we have the estimate
\[
\sum_{q\sim Q} c_q\sum_{m\sim M}\alpha_m\sum_{d\sim D}\beta_d\sum_{n'}f\Bigl(\frac{n'}{N'}\Bigr)\Bigl(\mathbf{1}_{m n' d\equiv a\Mod{q}}-\frac{\mathbf{1}_{(m n' d,q)=1}}{\phi(q)}\Bigr)\ll_{A,C} \frac{x}{(\log{x})^A}.
\]
Then for any $B>0$ and every interval $\mathcal{I}\subseteq [N,2N]$ we have
\[
\sum_{q\sim Q}c_q \sum_{m\sim M}\alpha_m\sum_{\substack{n\in\mathcal{I}\\ P^-(n)>z}}\Bigl(\mathbf{1}_{mn\equiv a\Mod{q}}-\frac{\mathbf{1}_{(m n,q)=1}}{\phi(q)}\Bigr)\ll_{B} \frac{x}{(\log{x})^B}.
\]
\end{lmm}
\begin{proof}
This is \cite[Lemma 18.2]{May1}.
\end{proof}
%
%
%
%
%%%%%%%%%%%%%%%%%%%%%%%%%%%%%%%%%%%%%%%%%%%%%%%%%%%%%%%%%%
%
%
%
%
\begin{lmm}[Deshouillers-Iwaniec estimate]\label{lmm:DeshouillersIwaniec}
Let $b_{n,r,s}$ be a 1-bounded sequence and $R,S,N,D,C\ll x^{O(1)}$. Let $g(c,d)=g_0(c/C,d/D)$ where $g_0$ is a smooth function supported on $[1/2,5/2]\times [1/2,5/2]$. Then we have
\[
\sum_{r\sim R} \sum_{\substack{s\sim S\\ (r,s)=1}}\sum_{n\sim N}b_{n,r,s}\sum_{d\sim D}\sum_{\substack{c\sim C\\ (rd,sc)=1}}g(c,d) e\Bigl(\frac{n\overline{dr}}{cs}\Bigr)\ll_{g_0} x^\epsilon \Bigl(\sum_{r\sim R}\sum_{s\sim S}\sum_{n\sim N}|b_{n,r,s}|^2\Bigr)^{1/2}\mathscr{J}.
\]
where
\[
\mathscr{J}^2=CS(RS+N)(C+DR)+C^2 D S\sqrt{(RS+N)R}+D^2NR.
\]
\end{lmm}
\begin{proof}
This is \cite[Theorem 12]{DeshouillersIwaniec} (correcting a minor typo in the last term of $\mathscr{J}^2$).
\end{proof}
%
%
%
%
%%%%%%%%%%%%%%%%%%%%%%%%%%%%%%%%%%%%%%%%%%%%%%%%%%%%%%%%%%
%
%
%
%
\section{Double divisor function estimates}\label{sec:DoubleDivisor}
%
%
%
%
%%%%%%%%%%%%%%%%%%%%%%%%%%%%%%%%%%%%%%%%%%%%%%%%%%%%%%%%%%
%
%
%
%
In this section we establish Proposition \ref{prpstn:DoubleDivisor}, which is a quick consequence of the Weil bound and the fundamental lemma of sieves. Although well-known, we give a full argument for completeness (it might also help the reader motivate \cite[Section 19]{May1} on the triple divisor function). These estimates are not a bottleneck for our results, and in fact several much stronger results could be used here (see, for example \cite{FouvryIwaniecDivisor}).
%
%
%
%
%%%%%%%%%%%%%%%%%%%%%%%%%%%%%%%%%%%%%%%%%%%%%%%%%%%%%%%%%%
%
%
%
%
\begin{lmm}[Smoothed divisor function estimate]\label{lmm:DoubleDivisor}
Let $N_1,N_2,M,Q\ge 1$ satisfy $x^{2\epsilon}\le N_1\le N_2$,  $N_1N_2M\asymp x$ and
\begin{align*}
Q&\le \frac{x^{2/3-2\epsilon} }{M^{2/3}}.
\end{align*}
Let $\psi_1$ and $\psi_2$ be smooth functions supported on $[1/2,5/2]$ satisfying $\psi_1^{(j)},\psi_2^{(j)}\ll_j (\log{x})^{j C}$ and let $\alpha_m$ be  a 1-bounded complex sequence. Let
\[
\mathscr{K}:=\sup_{\substack{(a,q)=1\\ q\sim Q}}\Bigl|\sum_{m\sim M}\alpha_m\sum_{n_1,n_2}\psi_1\Bigl(\frac{n_1}{N_1}\Bigr)\psi_2\Bigl(\frac{n_2}{N_2}\Bigr)\Bigl(\mathbf{1}_{mn_1n_2\equiv a\Mod{q}}-\frac{\mathbf{1}_{(m n_1n_2,q)=1}}{\phi(q)}\Bigr)\Bigr|.
\]
Then we have
\[
\mathscr{K}\ll_C \frac{x^{1-\epsilon}}{Q}.
\]
\end{lmm}
%
%
%
%
%%%%%%%%%%%%%%%%%%%%%%%%%%%%%%%%%%%%%%%%%%%%%%%%%%%%%%%%%%
%
%
%
%
(It is unimportant for this paper that Proposition \ref{prpstn:DoubleDivisor} holds pointwise for $q$ and uniformly over all $(a,q)=1$, but the proof is no harder.)
%
%
%
%
%%%%%%%%%%%%%%%%%%%%%%%%%%%%%%%%%%%%%%%%%%%%%%%%%%%%%%%%%%
%
%
%
%
\begin{proof}
Let the supremum occur at $a$ and $q$. We have that $\mathscr{K}=\mathscr{K}_{2}-\mathscr{K}_{1}$, where
\begin{align*}
\mathscr{K}_{1}&:=\frac{1}{\phi(q)}\sum_{\substack{m\sim M\\ (m,q)=1}}\alpha_m\sum_{\substack{n_1,n_2\\ (m n_1 n_2,q)=1}}\psi_1\Bigl(\frac{n_1}{N_1}\Bigr)\psi_2\Bigl(\frac{n_2}{N_2}\Bigr),\\
\mathscr{K}_{2}&:=\sum_{\substack{m\sim M\\ (m,q)=1}}\alpha_m\sum_{(n_2,q)=1}\psi_2\Bigl(\frac{n_2}{N_2}\Bigr)\sum_{\substack{n_1\\ n_1\equiv a\overline{m n_2}\Mod{q}}}\psi_1\Bigl(\frac{n_1}{N_1}\Bigr).
\end{align*}
By Lemma \ref{lmm:TrivialCompletion}, since $N_1\le N_2$ we have
\[
\sum_{\substack{n_1,n_2\\ (m n_1n_2,q)=1}}\psi_1\Bigl(\frac{n_1}{N_1}\Bigr)\psi_2\Bigl(\frac{n_2}{N_2}\Bigr)=\frac{\phi(q)^2}{q^2}N_1 N_2 \hat{\psi_1}(0)\hat{\psi_2}(0)+O(N_2 x^{o(1)}).
\]
This implies that
\[
\mathscr{K}_{1}=\mathscr{K}_{MT}+O\Bigl(\frac{x^{1+o(1)}}{Q N_1}\Bigr),
\]
where
\[
\mathscr{K}_{MT}:=N_1 N_2  \hat{\psi_1}(0)\hat{\psi_2}(0)\frac{\phi(q)}{q^2}\sum_{\substack{m\sim M\\ (m,q)=1}}\alpha_m.
\]
By Lemma \ref{lmm:Completion} we have that for $H_1:=x^\epsilon Q/N_1$
\[
\sum_{\substack{n_1\\ n_1\equiv a\overline{m n_2}\Mod{q}}}\psi_1\Bigl(\frac{n_1}{N_1}\Bigr)=\frac{N_1}{q}\hat{\psi}(0)+\frac{N_1}{q}\sum_{1\le |h_1|\le H_1}\hat{\psi_1}\Bigl(\frac{h_1 N_1}{q}\Bigr)e\Bigl(\frac{a h_1 \overline{m n_2}}{q}\Bigr)+O(x^{-10}).
\]
The final term makes a negligible contribution to $\mathscr{K}_{2}$. By Lemma \ref{lmm:TrivialCompletion}, the first term contributes to $\mathscr{K}_{2}$ a total
\[
\frac{N_1 \hat{\psi_1}(0)}{q}\sum_{\substack{m\sim M\\ (m,q)=1}}\alpha_m \sum_{(n_2,q)=1}\psi_2\Bigl(\frac{n_2}{N_2}\Bigr)=\mathscr{K}_{MT}+O\Bigl(\frac{x^{1+o(1)}}{Q N_2}\Bigr).
\]
Finally, by another application of Lemma \ref{lmm:Completion} we have
\begin{align*}
\sum_{(n_2,q)=1}\psi_2\Bigl(\frac{n_2}{N_2}\Bigr)&e\Bigl(\frac{a h_1 \overline{m n_2}}{q}\Bigr)=\frac{N_2}{q}\hat{\psi_2}(0)\sum_{(b,q)=1}e\Bigl(\frac{ah_1 \overline{m b}}{q}\Bigr)\\
&+\frac{N_2}{q}\sum_{1\le |h_2|\le H_2}\hat{\psi_2}\Bigl(\frac{ h_2 N_2}{q}\Bigr)\sum_{(b,q)=1}e\Bigl(\frac{a h_1 \overline{m b}+h_2 b}{q}\Bigr)+O(x^{-10}).
\end{align*}
The inner sum in the first term is a Ramanujan sum and so of size $O((h_1,q))$. The inner sum in the second term is a Kloosterman sum, and so of size $O(q^{1/2+o(1)}(h_1,h_2,q))$. The final term contributes a negligible amount. Thus we see that these terms contribute a total
\begin{align*}
&\ll \frac{N_1 N_2}{Q^2}\sum_{m\sim M}\sum_{\substack{1\le |h_1|\le H_1}}(h_1,q)+\frac{x^{o(1)}N_1 N_2}{Q^{3/2}}\sum_{m\sim M}\sum_{\substack{1\le |h_1|\le H_1\\ 1\le |h_2|\le H_2}}(h_1,h_2,q)\\
&\ll \frac{x^{o(1)} N_1 N_2 M H_1}{Q^2}+\frac{x^{o(1)} N_1 N_2 M H_1 H_2}{Q^{3/2}}\\
&\ll \frac{x^{1+o(1)}}{Q N_1}+x^{o(1)}M Q^{1/2}.
\end{align*}
Putting this together, we obtain
\[
\mathscr{K}\ll \frac{x^{1+o(1)}}{Q N_1}+x^{o(1)} M Q^{1/2}.
\]
This gives the result provided
\begin{align}
x^{2\epsilon}&\le N_1,\\
Q&\le \frac{x^{2/3-2\epsilon}}{M^{2/3}}.
\end{align}
This gives the result.
\end{proof}
%
%
%
%
%%%%%%%%%%%%%%%%%%%%%%%%%%%%%%%%%%%%%%%%%%%%%%%%%%%%%%%%%%
%
%
%
%
\begin{proof}[Proof of Proposition \ref{prpstn:DoubleDivisor}]
First we note that by Lemma \ref{lmm:Divisor} and the trivial bound, those $m$ with $|\alpha_m|\ge(\log{x})^C$ contribute a total $\ll x(\log{x})^{O_{B_0}(1)-C}$. This is negligible if $C=C(A,B_0)$ is large enough so, by dividing through by $(\log{x})^{C}$ and considering $A+C$ in place of $A$, it suffices to show the result when $|\alpha_m|\le 1$.

We apply Lemma \ref{lmm:Separation} to remove the condition $m n_1n_2\in\mathcal{I}$. Thus it suffices to show for every $B>0$ and every choice of interval $\mathcal{I}_M\subseteq[M,2M]$, $\mathcal{I}_1\subseteq[N_1,2N_1]$ and $\mathcal{I}_2\subseteq[N_2,2N_2]$ that we have
\[
\sum_{q\sim Q}\Bigl|\sum_{\substack{n_1\in \mathcal{I}_1\\ P^-(n)\ge z_0}}\sum_{\substack{n_2\in \mathcal{I}_2\\ P^-(n)\ge z_0}}\sum_{\substack{m\in\mathcal{I}_M }}\alpha_m\Bigl(\mathbf{1}_{m n_1n_2\equiv a\Mod{q}}-\frac{\mathbf{1}_{(mn_1n_2,q)=1}}{\phi(q)}\Bigr)\Bigr|\ll_B \frac{x}{(\log{x})^B}.
\]
We now remove the absolute values by inserting 1-bounded coefficients $c_q$. By two applications of Lemma \ref{lmm:SmoothReduction} with $z=z_0$, then see that it is sufficient to show that for every $A,C>0$, every choice of smooth functions $f_1,f_2$ supported on $[1/2,5/2]$ with $f_i^{(j)}\ll_j (\log{x})^{C j}$ and for every 1-bounded sequence $\beta_{d_1,d_2}$ and for every choice of $D_1,D_2,N_1',N_2'$ with $D_1,D_2\le y_0$ and $N_1'D_1\asymp N_1$, $N_2'D_2\asymp N_2$ we have that
\begin{align*}
&\sum_{q\sim Q} c_q\sum_{\substack{d_1\sim D_1\\ d_2\sim D_2}}\beta_{d_1,d_2}\sum_{n'_1,n_2'}f_1\Bigl(\frac{n_1'}{N'_1}\Bigr)f_2\Bigl(\frac{n_2'}{N_2'}\Bigr)\\
&\qquad\times\sum_{m\in \mathcal{I}_M}\alpha_m \Bigl(\mathbf{1}_{m n'_1 n_2' d_1 d_2\equiv a\Mod{q}}-\frac{\mathbf{1}_{(m n_1' n_2' d_1 d_2,q)=1}}{\phi(q)}\Bigr)\ll_{A,C} \frac{x}{(\log{x})^A}.
\end{align*}
Grouping together $m,d_1,d_2$, we see that Lemma \ref{lmm:DoubleDivisor} now gives the result, recalling that $D_1,D_2\le y_0=x^{o(1)}$ so $N_1'=N_1x^{-o(1)}\ge x^{2\epsilon}$ and $N_2'=N_2x^{-o(1)}\ge x^{2\epsilon}$ and $Q\le x^{2/3-3\epsilon}/M^{2/3}\le x^{2/3-2\epsilon}/(D_1D_2M)^{2/3}$.

An identical argument works if the summand is multiplied by $\log{n_1}$, since this just slightly adjusts the smooth functions appearing.
\end{proof}
%
%
%
%
%%%%%%%%%%%%%%%%%%%%%%%%%%%%%%%%%%%%%%%%%%%%%%%%%%%%%%%%%%
%
%
%
%
\section{Well-factorable estimates}\label{sec:WellFactorable}
%
%
%
%
%%%%%%%%%%%%%%%%%%%%%%%%%%%%%%%%%%%%%%%%%%%%%%%%%%%%%%%%%%
%
%
%
%
In this section we establish Proposition \ref{prpstn:WellFactorable}, which is the key result behind Theorem \ref{thrm:Factorable}. This can be viewed as a refinement of \cite[Theorem 1]{BFI1}. Indeed, Proposition \ref{prpstn:WellFactorable} essentially includes \cite[Theorem 1]{BFI1} as the special case $R=1$. The key advantage in our setup is to make use of the additional flexibility afforded by having a third factor available when manipulating the exponential sums. The argument does not have a specific regime when it is weakest; the critical case for Theorem \ref{thrm:Factorable} is the whole range $x^{1/10}\le N\le x^{1/3}$. (The terms with $N\le x^{1/10}$ or $N>x^{1/3}$ can be handled by a combination of the result for $N\in[x^{1/10},x^{1/3}]$ and Proposition \ref{prpstn:DoubleDivisor}.)
%
%
%
%
%%%%%%%%%%%%%%%%%%%%%%%%%%%%%%%%%%%%%%%%%%%%%%%%%%%%%%%%%%
%
%
%
%
\begin{lmm}[Well-factorable exponential sum estimate]\label{lmm:Factorable}
Let $Q'\le 2Q$, $H'\le x^{o(1)} QR^2 S^2/M$, $NM\asymp x$ and 
\begin{align}
%Q&<N,\\
N^2 R^2 S&< x^{1-7\epsilon},\\
N^2 R^3 S^4 Q&<x^{2-14\epsilon},\\
N R^2 S^5 Q&<x^{2-14\epsilon}.
\end{align}
Let $\gamma_r,\lambda_s,\alpha_n$ be 1-bounded complex coefficients, and let
\begin{align*}
\mathscr{W}&:=\sum_{\substack{Q\le q\le Q'\\ (q,a)=1}}\sum_{\substack{r_1,r_2\sim R}}\sum_{\substack{s_1,s_2\sim S \\ (r_1s_1,a r_2s_2)=1\\ (r_2s_2,a q d r_1 s_1)=1\\ r_1s_1\le B_1\\ r_2s_2\le B_2}}\frac{\gamma_{r_1}\lambda_{s_1}\overline{\gamma_{r_2}\lambda_{s_2}}}{r_1r_2s_1s_2q}\sum_{\substack{n_1,n_2\sim N \\ n_1\equiv n_2\Mod{q d}\\ (n_1,n_2 q d r_1 s_1)=1\\ (n_2,n_1 q d r_2 s_2)=1\\ (n_1r_2s_2,n_2)\in\mathcal{N}\\ |n_1-n_2|\ge N/(\log{x})^C }}\alpha_{n_1}\overline{\alpha_{n_2}}\\
&\qquad\times\sum_{1\le |h| \le H'}e\Bigl(\frac{ah(n_1-n_2)\overline{n_2 r_1 s_1 d q}}{ n_1 r_2s_2}\Bigr)
\end{align*}
for some $(d,a)=1$ where $\mathcal{N}$ is a set with the property that if $(a,b)\in\mathcal{N}$ and $(a',b')\in\mathcal{N}$ then $\gcd(a,b')=\gcd(a',b)=1$.

Then we have
\[
\mathscr{W}\ll \frac{N^2}{Q x^\epsilon}.
\]
\end{lmm}
\begin{proof}
We first make a change of variables. Since we have $n_1\equiv n_2\Mod{q d}$, we let $f d q=n_1-n_2$ for some integer $|f|\le 2N/d Q\le 2N/Q$, and we wish to replace $q$ with $(n_1-n_2)/d f$. We see that 
\[
(n_1-n_2)\overline{d q}=f\Mod{n_1 r_2 s_2}.
\]
Thus the exponential simplifies to
\[
e\Bigl(\frac{ah f\overline{r_1s_1n_2}}{n_1r_2s_2}\Bigr).
\]
The conditions $(n_1,n_2)=1$ and $n_1\equiv n_2\Mod{d q}$ automatically imply $(n_1n_2,d q)=1$, and so we find
\begin{align*}
\mathscr{W}&=\sum_{1\le |f|\le 2N/Q}\sum_{\substack{r_1,r_2\sim R\\ (r_1r_2,a)=1}}\sum_{\substack{s_2\sim S\\ (r_2s_2,a d r_1)=1\\ r_2s_2\le B_2}}\sideset{}{'}\sum_{\substack{n_1,n_2\sim N\\ n_1\equiv n_2\Mod{d f}}}\frac{\gamma_{r_1}\overline{\gamma_{r_2}\lambda_{s_2}}d f}{r_1r_2 s_2(n_1-n_2)}\\
&\qquad\times\sum_{\substack{s_1\sim S\\ (s_1,a n_1 r_2 s_2)=1\\ r_1s_1\le B_1}}\frac{\lambda_{s_1}}{s_1}\sum_{1\le |h|\le H'}\alpha_{n_1}\overline{\alpha_{n_2}}e\Bigl(\frac{a h f\overline{r_1s_1n_2}}{n_1r_2s_2}\Bigr).
\end{align*}
Here we have used $\sum'$ to denote that fact that we have suppressed the conditions
\begin{align*} &(n_1,n_2 r_1 s_1)=1,&\quad &(n_2,n_1 r_2s_2)=1,&\quad& (n_1r_2s_2,n_2)\in\mathcal{N},&\\
 &|n_1-n_2|\ge N/(\log{x})^C,& \quad& ((n_1-n_2)/d f,a r_2 s_2)=1,&\quad &Q d f\le n_1-n_2\le Q' d f. &
\end{align*}
We first remove the dependency between $r_1$ and $s_1$ from the constraint $r_1s_1\le B_1$ by noting 
\begin{align*}
\mathbf{1}_{r_1s_1\le B_1}&=\int_0^1\Bigl(\sum_{j\le B_1/r_1}e(-j\theta)\Bigr)e(s_1\theta)d\theta\\
&=\int_0^1 c_{r_1,\theta} \min\Bigl(\frac{B_1}{R},|\theta|^{-1}\Bigr)e(s_1\theta)d\theta
\end{align*}
for some 1-bounded coefficients $c_{r_1,\theta}$. Thus
\begin{align*}
\mathscr{W}&=\int_0^1\min\Bigl(\frac{B_1}{R},|\theta|^{-1}\Bigr)\mathscr{W}_2(\theta)d\theta\ll (\log{x}) \sup_{\theta}|\mathscr{W}_2(\theta)|,
\end{align*}
where $\mathscr{W}_2=\mathscr{W}_2(\theta)$ is given by
\begin{align*}
\mathscr{W}_2&:=\sum_{1\le |f|\le 2N/Q}\sum_{\substack{r_1,r_2\sim R\\ (r_1r_2,a)=1}}\sum_{\substack{s_2\sim S\\ (r_2s_2,a d r_1)=1\\ r_2s_2\le B_2}}\sideset{}{'}\sum_{\substack{n_1,n_2\sim N\\ n_1\equiv n_2\Mod{d f}}}\frac{\gamma_{r_1}c_{r_1,\theta}\overline{\gamma_{r_2}\lambda_{s_2}}d f}{r_1r_2 s_2(n_1-n_2)}\\
&\qquad\times\sum_{\substack{s_1\sim S\\ (s_1,an_1r_2s_2)=1}}\frac{e(s_1\theta)\lambda_{s_1}}{s_1}\sum_{1\le |h|\le H'}\alpha_{n_1}\overline{\alpha_{n_2}}e\Bigl(\frac{a h f\overline{r_1s_1n_2}}{n_1r_2s_2}\Bigr).
\end{align*}
In order to show $\mathscr{W}\ll N^2/(Qx^\epsilon)$ we see it is sufficient to show $\mathscr{W}_2\ll N^2/(Q x^{2\epsilon})$. We now apply Cauchy-Schwarz in the $f$, $n_1$, $n_2$, $r_1$, $r_2$ and $s_2$ variables. This gives
\begin{align*}
\mathscr{W}_2\ll \frac{N R S^{1/2}(\log{x})^2}{Q R^2 S^2}\mathscr{W}_3^{1/2},
\end{align*}
where
\begin{align*}
\mathscr{W}_3&:=\sum_{1\le |f|\le 2N/Q}\sum_{\substack{n_1,n_2\sim N\\ n_1\equiv n_2\Mod{d f}}}\sum_{r_1,r_2\sim R}\\
&\qquad\times\sum_{\substack{s_2\sim S\\ (n_2r_1,n_1r_2s_2)=1}}\Bigl|\sum_{\substack{s_1\sim S\\ (s_1,an_1r_2s_2)=1}}\sum_{1<|h|\le H'}\lambda_{s_1}' e\Bigl(\frac{ah f\overline{r_1s_1n_2}}{n_1r_2s_2}\Bigr)\Bigr|^2,
\end{align*}
and where
\[
\lambda_s':=\frac{S}{s}\lambda_{s}e(s\theta)
\]
are 1-bounded coefficients. Note that we have dropped many of the constraints on the summation for an upper bound. In order to show that $\mathscr{W}_2\ll N^2/(Q x^{2\epsilon})$ we see it is sufficient to show that $\mathscr{W}_3\ll N^2 R^2 S^3/x^{5\epsilon}$. We first drop the congruence condition on $n_1,n_2\Mod{d f}$ for an upper bound, and then we combine $n_2r_1$ into a single variable $b$ and $n_1 r_2 s_2$ into a single variable $c$. Using the divisor bound to control the number of representations of $c$ and $b$, and inserting a smooth majorant, this gives
\begin{align*}
\mathscr{W}_3&\le x^{o(1)}\sup_{\substack{ B \ll N R\\ C\ll N R S \\ F\ll N/Q}} \mathscr{W}_4,
\end{align*}
where
\begin{align*}
\mathscr{W}_4&:=\sum_{b}\sum_{\substack{c\\ (b,c)=1}}g(b,c)\sum_{f\sim F}\Bigl|\sum_{\substack{s_1\sim S\\ (s_1,a c)=1}}\sum_{1<|h|\le H'}\lambda_{s_1}' e\Bigl(\frac{a h f\overline{b s_1}}{c}\Bigr)\Bigr|^2\\
g(b,c)&:=\psi_0\Bigl(\frac{b}{B}\Bigr)\psi_0\Bigl(\frac{c}{C}\Bigr).
\end{align*}
In order to show $\mathscr{W}_3\ll N^2 R^2 S^3/x^{5\epsilon}$, it is sufficient to show that
\begin{equation}
\mathscr{W}_4\ll \frac{N^2R^2 S^3}{x^{6\epsilon}}.\label{eq:W4Target}
\end{equation}
We expand the square and swap the order of summation, giving
\[
\mathscr{W}_4=\sum_{\substack{s_1,s_2\sim S\\ (s_1s_2,a)=1}}\sum_{1<|h_1|,|h_2|\le H'}\lambda_{s_1}'\overline{\lambda_{s_2}'}\sum_{b}\sum_{f\sim F}\sum_{\substack{c\\ (c,b s_1s_2)=1}}g(b,c)e\Bigl(a f\ell\frac{\overline{b s_1s_2}}{c}\Bigr),
\]
where
\[
\ell=h_1s_1-h_2s_2.
\]
We now split the sum according to whether $\ell=0$ or not.
\[
\mathscr{W}_4=\mathscr{W}_{\ell=0}+\mathscr{W}_{\ell\ne 0}.
\]
To show \eqref{eq:W4Target} it is sufficient to show
\begin{equation}
\mathscr{W}_{\ell=0}\ll \frac{N^2R^2 S^3}{x^{6\epsilon}}\qquad\text{and}\qquad \mathscr{W}_{\ell\ne 0}\ll \frac{N^2R^2 S^3}{x^{6\epsilon}}.\label{eq:WTargets}
\end{equation}
We first consider $\mathscr{W}_{\ell=0}$, and so terms with $h_1s_1=h_2s_2$. Given $h_1,s_1$ there are at most $x^{o(1)}$ choices of $h_2,s_2$, and so at most $x^{o(1)}HS$ choices of $h_1,h_2,s_1,s_2$. Thus we see that
\begin{align*}
\mathscr{W}_{\ell=0} \ll x^{o(1)}H S B F C&\ll x^{o(1)}\frac{R^2 S^2 Q}{M}\cdot S\cdot N R\cdot \frac{N}{Q}\cdot N R S\\
&\ll \frac{N^4 R^4 S^4}{x^{1-\epsilon}}.\label{eq:WDiag}
\end{align*}
This gives an acceptably small contribution for \eqref{eq:WTargets} provided
\[
\frac{N^4 R^4 S^4}{x^{1-\epsilon}}\ll \frac{N^2 R^2 S^3}{x^{6\epsilon}},
\]
which rearranges to
\begin{equation}
N^2 R^2 S\ll x^{1-7\epsilon}.\label{eq:FactorCond1}
\end{equation}
We now consider $\mathscr{W}_{\ell\ne 0}$. We let $y=a f(h_1 s_1-h_2s_2)\ll x^{o(1)} N R^2 S^3/M$ and $z=s_1 s_2\ll S^2$. Putting these variables in dyadic intervals and using the symmetry between $y$ and $-y$, we see that
\[
\mathscr{W}_{\ell\ne 0}\ll \log{x}\sum_{z\sim Z}\sum_{y\sim Y}b_{z,y}\Bigl|\sum_{b}\sum_{\substack{c\\ (c,z b)=1}}g(b,c)e\Bigl(\frac{y\overline{z b}}{c}\Bigr)\Bigr|,
\]
where $Z\asymp S^2$, $Y\ll x^{o(1)} N R^2 S^3/M$ and
\[
b_{z,y}=\mathop{\sum_{s_1,s_2\sim S}\sum_{1\le |h_1|,|h_2|\le H'}\sum_{f\sim F}}\limits_{\substack{s_1 s_2=z\\ a f(h_1s_1-h_2s_2)=y}}1.
\]
By Lemma \ref{lmm:DeshouillersIwaniec} we have that
\begin{equation}
\mathscr{W}_{\ell\ne 0}\ll x^\epsilon \Bigl(\sum_{z\sim Z}\sum_{y\sim Y}b_{z,y}^2\Bigr)^{1/2}\mathscr{J},
\label{eq:WOffDiag1}
\end{equation}
where
\begin{equation}
\mathscr{J}^2\ll C(Z+Y)(C+B Z)+C^2B\sqrt{ (Z+Y)Z}+B^2 Y Z.
\label{eq:JBound}
\end{equation}
We first consider the $b_{z,y}$ terms. We note that given a choice of $z,y$ there are $x^{o(1)}$ choices of $s_1,s_2,k,f$ with $z=s_1s_2$ and $y=a k f$ by the divisor bound. Thus by Cauchy-Schwarz, we see that
\begin{align*}
\sum_{z\sim Z}\sum_{y\sim Y}b_{z,y}^2&=\sum_{z\sim Z}\sum_{y\sim Y}\Bigl(\sum_{\substack{s_1,s_2\sim  S\\ s_1 s_2=z}}\sum_{\substack{f\sim F}}\sum_{\substack{ k\\ ak f=y}}\sum_{\substack{1\le |h_1|,|h_2|\ll H \\ h_1 s_1-h_2s_2=k}}1\Bigr)^2\\
&\ll x^{o(1)}\sum_{s_1,s_2\sim  S}\sum_{f\sim F}\sum_{k}\Bigl(\sum_{\substack{1\le |h_1|,|h_2|\ll H\\ h_1s_1-h_2s_2=k}}1\Bigr)^2\\
&\ll x^{o(1)} F \sum_{s_1,s_2\sim S}\sum_{\substack{ 1\le |h_1|,|h_1'|,|h_2|,|h_2'|\ll H \\ (h_1-h_1')s_1=(h_2-h_2')s_2}}1
\end{align*}
We consider the inner sum. If $(h_1-h_1')s_1=0=(h_2-h_2')s_2$ we must have $h_1=h_1'$, $h_2=h_2'$, so there are $O(H^2S^2)$ choices $h_1,h_1',h_2,h_2',s_1,s_2$. If instead $(h_1-h_1')s_1=(h_2-h_2')s_2\ne 0$ there are $O(H S)$ choices of $t=(h_1-h_1')s_1\ne 0$. Given a choice of $t$, there are $x^{o(1)}$ choices of $s_1,s_2,h_1-h_1',h_2-h_2'$. Thus there are $O(H^3S)$ choices $h_1,h_1',h_2,h_2',s_1,s_2$ with $(h_1-h_1')s_1=(h_2-h_2')s_2\ne 0$. Thus
\begin{align*}
\sum_{z\sim Z}\sum_{y\sim Y}b_{z,y}^2&\ll x^{o(1)}  \frac{N}{Q} (H^2 S^2+H^3 S)\\
&\ll x^\epsilon\Bigl(\frac{N R^4 S^6 Q}{M^2}+\frac{N R^6 S^7 Q^2}{M^3}\Bigr).
\end{align*}
In particular, since we are assuming that $N^2R^2S<x^{1-\epsilon}\le MN$ and $N>Q$, we have 
\[
M>R^2S N>R^2 S Q.
\]
Thus this simplifies to give
\begin{equation}
\sum_{z\sim Z}\sum_{y\sim Y}b_{z,y}^2\ll x^\epsilon \frac{N R^4 S^6 Q}{M^2}.
\label{eq:OffDiagSeq}
\end{equation}
We now consider $\mathscr{J}$. Since the bound \eqref{eq:JBound} is increasing and polynomial in $C,B,Z,Y$, the maximal value is at most $x^{o(1)}$ times the value when $C=N R S$, $Z=S^2$, $Y=N R^2 S^3/M$ and $B=N R$, and so it suffices to consider this case. We note that our bound $M>R^2 S N$ from \eqref{eq:FactorCond1} then implies that that $Z>Y$, and so, noting that $B Z>C$ and $C^2 B Z > B^2 Y Z$, this simplifies our bound for $\mathscr{J}$ to
\begin{align}
\mathscr{J}^2&\ll x^{o(1)}(C B Z^2+C^2 B Z+B^2 Y Z)\nonumber\\
&\ll x^{\epsilon}(C B Z^2+C^2 B Z)\nonumber\\
&=x^{\epsilon} N^2 R^2 S^5+x^\epsilon N^3 R^3 S^4.
\label{eq:FactorJBound}
\end{align}
Putting together \eqref{eq:WOffDiag1}, \eqref{eq:OffDiagSeq} and \eqref{eq:FactorJBound}, we obtain
\begin{align*}
\mathscr{W}_{\ell\ne 0}&\ll x^\epsilon \Bigl(x^\epsilon \frac{N R^4 S^6 Q}{M^2}\Bigr)^{1/2}\Bigl(x^\epsilon N^2 R^2 S^5+ x^\epsilon N^3 R^3 S^4\Bigr)^{1/2}\\
&\ll x^{2\epsilon}\Bigl(\frac{N^3 R^6 S^{11} Q}{M^2}+\frac{N^4 R^7 S^{10}Q}{M^2}\Bigr)^{1/2}.
\end{align*}
Thus we obtain \eqref{eq:WTargets} if
\[
N^3 R^6 S^{11}Q+N^4 R^7 S^{10}Q<x^{-14\epsilon} N^4 R^4 S^6 M^2.
\]
recalling $NM\asymp x$, we see that this occurs if we have
\begin{align}
N R^2 S^5 Q&<x^{2-14\epsilon},\\
N^2 R^3 S^4 Q&<x^{2-14\epsilon}.
\end{align}
This gives the result.
\end{proof}
%
%
%
%
%%%%%%%%%%%%%%%%%%%%%%%%%%%%%%%%%%%%%%%%%%%%%%%%%%%%%%%%%%
%
%
%
%
\begin{prpstn}[Well-factorable estimate for convolutions]\label{prpstn:MainProp}
Let $NM\asymp x$ and $Q_1,Q_2,Q_3$ satisfy
\begin{align*}
Q_1&<\frac{N}{x^\epsilon},\\
N^2 Q_2 Q_3^2&<x^{1-8\epsilon},\\
N^2 Q_1 Q_2^4 Q_3^3&<x^{2-15\epsilon},\\
N Q_1 Q_2^5 Q_3^2&<x^{2-15\epsilon}.
\end{align*}
Let $\alpha_n,\beta_m$ be $1$-bounded complex sequences such that $\alpha_n$ satisfies the Siegel-Walfisz condition \eqref{eq:SiegelWalfisz} and $\alpha_n$ is supported on $n$ with all prime factors bigger than $z_0=x^{1/(\log\log{x})^3}$. Let $\gamma_{q_1},\lambda_{q_2},\nu_{q_3}$ be 1-bounded complex coefficients supported on $(q_i,a)=1$ for $i\in\{1,2,3\}$. Let
\[
\Delta(q):=\sum_{n\sim N}\alpha_n\sum_{m\sim M}\beta_m\Bigl(\mathbf{1}_{nm\equiv a\Mod{q}}-\frac{\mathbf{1}_{(n m,q)=1}}{\phi(q)}\Bigr).
\]
Then for every $A>0$ we have
\[
\sum_{q_1\sim Q_1}\sum_{q_2\sim Q_2}\sum_{q_3\sim Q_3}\gamma_{q_1}\lambda_{q_2}\nu_{q_3}\Delta(q_1q_2q_3)\ll_A\frac{x}{(\log{x})^A}.
\]
\end{prpstn}
\begin{proof}
First we factor $q_2=q_2'q_2''$ and $q_3=q_3'q_3''$ where $P^-(q_2'),P^-(q_3')>z_0\ge P^+(q_2''),P^+(q_3'')$ into parts with large and small prime factors. By putting these in dyadic intervals, we see that it suffices to show for every $A>0$ and every choice of $Q_2'Q_2''\asymp Q_2$, $Q_3'Q_3''\asymp Q_3$ that
 \begin{align*}
 &\sum_{q_1\sim Q_1}\sum_{\substack{q_2'\sim Q_2'\\ P^-(q_2')>z_0}}\sum_{\substack{q_2''\sim Q_2''\\ P^+(q_2'')\le z_0}}\sum_{\substack{q_3'\sim Q_3'\\ P^-(q_3')\ge z_0}}\sum_{\substack{q_3''\sim Q_3''\\ P^+(q_3'')\le z_0}}\gamma_{q_1}\lambda_{q_2'q_2''}\nu_{q_3'q_3''}\Delta(q_1q_2'q_2''q_3'q_3'')\ll_A\frac{x}{(\log{x})^A}.
 \end{align*}
By Lemma \ref{lmm:RoughModuli} we have the result unless $Q_2'',Q_3''\le y_0=x^{1/\log\log{x}}$. We let $d=q_2'' q_3''$ and define
\[
\lambda_{q,d,r}:=\mathbf{1}_{P^-(r)> z_0}\sum_{\substack{q_2''q_3''=d\\ q_1\sim Q_1\\ q_2''\sim Q_2''\\ q_3''\sim Q_3''\\ P^+(q_2''q_3'')\le z_0}}\,\,\sum_{\substack{q_2'q_3'=r\\ q_2'\sim Q_2'\\ q_3'\sim Q_3'}}\lambda_{q_2'q_2''}\nu_{q_3'q_3''}.
\]
We note that $\lambda_{q,d,r}$ doesn't depend on $q$. With this definition we see it suffices to show that for every $A>0$ and every choice of $D,R$ with $D R\asymp Q_2 Q_3$ and $D\le y_0^2$ we have that
\[
\sum_{q\sim Q_1}\sum_{d\sim D}\sum_{r\sim R}\gamma_{q}\lambda_{q,d,r}\Delta(q d r)\ll_A \frac{x}{(\log{x})^A}.
\]
We now apply Proposition \ref{prpstn:GeneralDispersion} (we may apply this since $N>Q_1 x^\epsilon>Q_1D(\log{x})^C$ and $N<x^{1-\epsilon}$ by assumption of the lemma). This shows that it suffices to show that
\[
|\mathscr{E}_1|+|\mathscr{E}_2|\ll \frac{N^2}{D Q_1 y_0},
\]
where
\begin{align*}
\mathscr{E}_{1}&:=\sum_{\substack{q\\ (q,a)=1}}\sum_{\substack{d\sim D\\ (d,a)=1}}\sum_{\substack{r_1,r_2\sim R\\ (r_1r_2,a)=1}}\psi_0\Bigl(\frac{q}{Q_1}\Bigr)\frac{\lambda_{q,d,r_1}\overline{\lambda_{q,d,r_2}} }{\phi(q d r_2)q d r_1}\sum_{\substack{n_1,n_2\sim N\\ (n_1,q d r_1)=1\\(n_2,q d  r_2)=1}}\alpha_{n_1}\overline{\alpha_{n_2}}\\
&\qquad \times\sum_{1\le |h|\le H_1}\hat{\psi}_0\Bigl(\frac{h M}{q d r_1}\Bigr)e\Bigl( \frac{a h \overline{ n_1}}{q d r_1}\Bigr),\\
\mathscr{E}_2&:=\sum_{\substack{q\\ (q,a)=1}}\psi_0\Bigl(\frac{q}{Q_1}\Bigr)\sum_{\substack{d\sim D\\ (d,a)=1}}\sum_{\substack{r_1,r_2\sim R\\ (r_1,a r_2)=1\\ (r_2,a q d r_1)=1}}\frac{\lambda_{q,d,r_1}\overline{\lambda_{q,d,r_2}}}{q d r_1 r_2}\sum_{\substack{n_1,n_2\sim N\\ n_1\equiv n_2\Mod{q d}\\ (n_1,n_2 q d r_1)=1\\(n_2,n_1 q d r_2)=1\\ |n_1-n_2|\ge N/(\log{x})^C}}\alpha_{n_1}\overline{\alpha_{n_2}}\\
&\qquad \times\sum_{1\le |h|\le H_2}\hat{\psi}_0\Bigl(\frac{h M}{q d r_1 r_2}\Bigr)e\Bigl(\frac{ah\overline{n_1r_2}}{q d r_1}+\frac{ah\overline{n_2 q d r_1}}{r_2}\Bigr),\\
H_1&:=\frac{Q D R}{M}\log^5{x},\\
H_2&:=\frac{Q D R^2}{M}\log^5{x}.
\end{align*}
Since $\lambda_{q,d,r}$ is independent of $q$, we may apply Lemma \ref{lmm:BFI2} to conclude that
\[
\mathscr{E}_1\ll \frac{N^2}{Q_1 D x^\epsilon},
\]
provided we have
\begin{align}
D R N^{3/2}&<x^{1-2\epsilon},\\
Q_1 D R&<x^{1-2\epsilon}.
\end{align}
These are both implied by the conditions of the lemma, recalling that $DR\asymp Q_2Q_3$. Thus it suffices to bound $\mathscr{E}_2$. Since $D\le y_0^2=x^{o(1)}$, it suffices to show 
\[
\mathscr{E}_3\ll \frac{N^2}{Q_1 x^{\epsilon/10}},
\]
for each $d\le y_0^2$, where $\mathscr{E}_3=\mathscr{E}_3(d)$ is given by
\begin{align*}
\mathscr{E}_3&:=\sum_{\substack{(q,a)=1}}\psi_0\Bigl(\frac{q}{Q_1}\Bigr)\sum_{\substack{r_1,r_2\sim R\\ (r_1,a r_2)=1\\ (r_2,a q d r_1)=1}}\frac{\lambda_{q,d,r_1}\overline{\lambda_{q,d,r_2}}}{q r_1 r_2}\sum_{\substack{n_1,n_2\sim N\\ n_1\equiv n_2\Mod{q d}\\ (n_1,n_2 q d r_1)=1\\(n_2,n_1 q d r_2)=1\\ |n_1-n_2|\ge N/(\log{x})^C}}\alpha_{n_1}\overline{\alpha_{n_2}}\\
&\qquad\qquad \times\sum_{1\le |h|\le H_2}\hat{\psi}_0\Bigl(\frac{h M}{q d r_1 r_2}\Bigr)e\Bigl(\frac{ah\overline{n_1r_2}}{q d r_1}+\frac{ah\overline{n_2 q d r_1}}{r_2}\Bigr).
\end{align*}
Since $\lambda_{q,d,r}$ is independent of $q$ and we treat each $d$ separately, we may suppress the $q,d$ dependence by writing $\lambda_{r}$ in place of $\lambda_{q,d,r}$. We now apply Lemma \ref{lmm:Simplification}. This shows it suffices to show that
\[
\mathscr{E}'\ll \frac{N^2}{Q_1 x^{\epsilon/2}},
\]
where
\[
\mathscr{E}':=\sum_{\substack{Q_1\le q\le Q_1'\\ (q,a)=1}}\sum_{\substack{R\le r_1\le  R_1\\ R\le r_2\le R_2\\ (r_1, a r_2)=1\\ (r_2,a q d r_1)=1}}\frac{\lambda_{r_1}\overline{\lambda_{r_2}}}{q d  r_1 r_2}\sum_{\substack{n_1,n_2\sim N\\ n_1\equiv n_2\Mod{q d}\\ (n_1,q d r_1n_2)=1\\ (n_2,q d r_2n_1)=1\\ (n_1r_2,n_2)\in\mathcal{N}}}\alpha_{n_1}\overline{\alpha_{n_2}}\sum_{1\le |h| \le H'} e\Bigl(\frac{ ah\overline{n_2 q d r_1}(n_1-n_2)}{n_1 r_2}\Bigr),
\]
and where $Q_1'\le 2Q_1$ and $R_1,R_2\le 2R$ and $H'\le H_2$.

We recall the definition of $\lambda_{q,d,r}$ and expand it as a sum. Since $d$ is fixed, there are $x^{o(1)}$ possible choices of $q_2'',q_3''$. Fixing one such choice, we then see $\mathscr{E}'$ is precisely of the form considered in Lemma \ref{lmm:Factorable}. This then gives the result, provided
\begin{align*}
Q_1&<\frac{N}{x^\epsilon},\\
N^2 Q_2' Q_3'{}^2 &<x^{1-7\epsilon},\\
N^2 Q_1 Q_2'{}^4 Q_3'{}^3&<x^{2-14\epsilon},\\
N Q_1 Q_2'{}^5 Q_3'{}^2&<x^{2-14\epsilon}.
\end{align*}
Since $Q_2'\le Q_2$ and $Q_3'\le Q_3$, these bounds follow from the assumptions of the lemma.
\end{proof}
%
%
%
%
%%%%%%%%%%%%%%%%%%%%%%%%%%%%%%%%%%%%%%%%%%%%%%%%%%%%%%%%%%
%
%
%
%
\begin{proof}[Proof of Proposition \ref{prpstn:WellFactorable}]
First we note that by Lemma \ref{lmm:Divisor} the set of $n,m$ with $\max(|\alpha_n|,|\beta_m|)\ge(\log{x})^C$ and $n m\equiv a\Mod{q}$ has size $\ll x(\log{x})^{O_{B_0}(1)-C}/q$, so these terms contribute negligibly if $C=C(A,B_0)$ is large enough. Thus, by dividing through by $(\log{x})^{2C}$ and considering $A+2C$ in place of $A$, it suffices to show the result when all the sequences are 1-bounded. ($\alpha_n$ still satisfies \eqref{eq:SiegelWalfisz} by Lemma \ref{lmm:SiegelWalfiszMaintain}.) The result follows from the Bombieri-Vinogradov Theorem if $Q\le x^{1/2-\epsilon}$, so we may assume that $Q\in[x^{1/2-\epsilon},x^{3/5-10\epsilon}]$.

We use Lemma \ref{lmm:Separation} to remove the condition $n m\in\mathcal{I}$, and see suffices to show for $B=B(A)$ sufficiently large in terms of $A$
\begin{equation}
\sum_{q\le x^{3/5-\epsilon}}\lambda_q\sum_{n\in\mathcal{I}_N}\alpha_n\sum_{\substack{m\in \mathcal{I}_M}}\beta_m\Bigl(\mathbf{1}_{n m\equiv a\Mod{q}}-\frac{\mathbf{1}_{(nm,q)=1}}{\phi(q)}\Bigr)\ll_B\frac{x}{(\log{x})^B}
\label{eq:FactorableTarget}
\end{equation}
uniformly over all intervals $\mathcal{I}_N\subseteq[N,2N]$ and $\mathcal{I}_M\subseteq[M,2M]$.

Let us define for $x^\epsilon\le N\le x^{2/5}$
\[
Q_1:=\frac{N}{x^\epsilon},\qquad Q_2:=\frac{Q}{x^{2/5-\epsilon}},\qquad Q_3:=\frac{x^{2/5}}{N}.
\]
We note that $Q_1Q_2Q_3=Q$ and $Q_1,Q_2,Q_3\ge 1$. Since $\lambda_q$ is triply well factorable of level $Q$, we can write
\begin{equation}
\lambda_{q}=\sum_{q_1q_2q_3=q}\gamma^{(1)}_{q_1}\gamma^{(2)}_{q_2}\gamma^{(3)}_{q_3},
\label{eq:LambdaFact}
\end{equation}
for some 1-bounded sequences $\gamma^{(1)},\gamma^{(2)},\gamma^{(3)}$ with $\gamma^{(i)}_q$ supported on $q\le Q_i$ for $i\in\{1,2,3\}$.

We now  substitute \eqref{eq:LambdaFact} into \eqref{eq:FactorableTarget} and put each of $q_1,q_2,q_3$ into one of $O(\log^3{x})$ dyadic intervals $(Q_1',2Q_1']$, $(Q_2',2Q_2']$ and $(Q_3',2Q_3']$ respectively. Since $Q_1'\le Q_1$, $Q_2'\le Q_2$ and $Q_3'\le Q_3$ and $Q_1Q_2Q_3=Q\le x^{3/5-10\epsilon}$ we have
\begin{align*}
Q_1'&\le\frac{N}{x^\epsilon},\\
N^2 Q_2' Q_3'{}^2 &\le Qx^{2/5+\epsilon} <x^{1-7\epsilon},\\
N^2 Q_1' Q_2'{}^4 Q_3'{}^3&\le \frac{Q^4}{x^{2/5-3\epsilon}} <x^{2-14\epsilon},\\
N Q_1' Q_2'{}^5 Q_3'{}^2&\le \frac{Q^5}{x^{6/5-4\epsilon}}<x^{2-14\epsilon}.
\end{align*}
Thus, we see that Proposition \ref{prpstn:MainProp} now gives the result.
\end{proof}
We have now established both Proposition \ref{prpstn:DoubleDivisor} and Proposition \ref{prpstn:WellFactorable}, and so completed the proof of Theorem \ref{thrm:Factorable}.
%
%
%
%
%%%%%%%%%%%%%%%%%%%%%%%%%%%%%%%%%%%%%%%%%%%%%%%%%%%%%%%%%%
%
%
%
%
\section{Proof of Theorem \ref{thrm:Linear}}\label{sec:Corollary}
\allowdisplaybreaks
First we recall some details of the construction of sieve weights associated to the linear sieve. We refer the reader to \cite{IwaniecFactorable} or \cite[Chapter 12.7]{Opera} for more details. The standard upper bound sieve weights $\lambda^+_d$ for the linear sieve of level $D$ are given by
\[
\lambda^+_d:=
\begin{cases}
\mu(d),\qquad &d\in\mathcal{D}^+(D),\\
0,&\text{otherwise,}
\end{cases}
\]
where
\[
\mathcal{D}^+(D):=\Bigl\{p_1\cdots p_r:\, p_1\ge p_2\ge \dots \ge p_r,\,\,p_1\cdots p_{2j}p_{2j+1}^3\le D\text{ for $0\le j<r/2$}\Bigr\}.
\]
Moreover, we recall the variant $\tilde{\lambda}_d^+$ of these sieve weights where one does not distinguish between the sizes of primes $p_j\in [D_j,D_j^{1+\eta}]$ with $D_j>x^\epsilon$ for some small constant $\eta>0$. (i.e. if $d=p_1\cdots p_r$ with $p_j\in [D_j,D_j^{1+\eta}]$ and $D_1\ge \dots \ge D_r\ge x^\epsilon$, then $\tilde{\lambda}^+_d=(-1)^r$ if $D_1\ge \dots \ge D_r$ and $D_1\cdots D_{2j}D_{2j+1}^3\le D^{1/(1+\eta)}$ for all $0\le j<r/2$, and otherwise $\tilde{\lambda}^+_d=0$.) This variant is a well-factorable function in the sense that for any choice of $D_1D_2=D$ we can write $\tilde{\lambda}^+=\sum_{1\le j\le \epsilon^{-1}}\alpha^{(j)}\star\beta^{(j)}$ where $\alpha^{(j)}_n$ is a sequence supported on $n\le D_1$ and $\beta^{(j)}_m$ is supported on $m\le D_2$. The construction of the sequence $\tilde{\lambda}_d^+$ follows from the fact that if $d\in\mathcal{D}^+(D)$ and $D=D_1D_2$ then $d=d_1d_2$ with $d_1\le D_1$ and $d_2\le D_2$. This produces essentially the same results as the original weights when combined with a fundamental lemma type sieve to remove prime factors less than $x^\epsilon$.

In view of Proposition \ref{prpstn:MainProp} and Proposition \ref{prpstn:DoubleDivisor}, in order to prove Theorem \ref{thrm:Linear} it suffices to construct a similar variant $\hat{\lambda}_d^+$ such that for every $N\in [x^\epsilon,x^{1/3+\epsilon}]$ we can write $\hat{\lambda}_d^+=\sum_{1\le j\le \epsilon^{-1}}\alpha^{(j)}\star\beta^{(j)}\star\gamma^{(j)}$ with $\alpha^{(j)}_n$ supported on $n\le D_1$ and $\beta^{(j)}_n$ supported on $n\le D_2$ and $\gamma^{(j)}_n$ supported on $n\le D_3$ for some choice of $D_1,D_2,D_3$ satisfying
\begin{align*}
D_1<\frac{N}{x^\epsilon},\quad 
N^2 D_2 D_3^2<x^{1-8\epsilon},\quad 
N^2 D_1 D_2^4 D_3^3<x^{2-15\epsilon},\quad
N D_1 D_2^5 D_3^2<x^{2-15\epsilon}.
\end{align*}
An identical argument to the construction of $\tilde{\lambda}_d^+$ shows that we can construct such a sequence $\hat{\lambda}_d^+$ if every $d\in \mathcal{D}^+(D)$ can be written as $d=d_1d_2d_3$ with $d_1\le D_1$, $d_2\le D_2$ and $d_3\le D_3$ satisfying the above constraints. Thus, in order to prove Theorem \ref{thrm:Linear} it suffices to establish the following result.
%
%
%
%
%%%%%%%%%%%%%%%%%%%%%%%%%%%%%%%%%%%%%%%%%%%%%%%%%%%%%%%%%%
%
%
%
%
\begin{prpstn}[Factorization of elements of $\mathcal{D}^+(D)$]\label{prpstn:Factorization}
Let $0<\delta<1/1000$ and let $D=x^{7/12-50\delta}$, $x^{2\delta} \le N\le x^{1/3+\delta/2}$ and $d\in\mathcal{D}^+(D)$. Then there is a factorization $d=d_1d_2d_3$ such that
\begin{align*}
d_1&\le\frac{N}{x^\delta},\\
N^2d_2d_3^2&\le x^{1-\delta},\\
N^2d_1d_2^4d_3^3&\le x^{2-\delta},\\
N d_1d_2^5d_3^2&\le x^{2-\delta}.
\end{align*}
\end{prpstn}
\begin{proof}
Let $d=p_1\cdots p_r\in\mathcal{D}^+$. We split the argument into several cases depending on the size of the factors.

\textbf{Case 1: $p_1\ge D^2/x^{1-3\delta}$.}

Let $D_1:=N x^{-\delta}$, $d_2:=D_2:=p_1$ and $D_3:=D x^\delta/(Np_1)$. We note that $D_1,D_2,D_3\ge 1$ from our bounds on $N$ and $p_1^3\le D$. Since $p_2\le p_1\le D^{1/3}$ and $D_1D_3=D/p_1$ we see that $p_2^2\le D_1D_3$, so either $p_2\le D_1$ or $p_2\le D_3$. Moreover, we see that since $p_1\cdots p_{j-1}p_j^2\le D$ for all $j\le r$, we have that $p_j^2\le D_1D_3/(p_2\cdots p_{j-1})$ for all $j\ge 3$. Thus, by considering $p_2$, $p_3$, $\dots$ in turn, we can greedily form products $d_1$ and $d_3$ with $d_1\le D_1$ and $d_3\le D_3$ and $d_1d_3=p_2\cdots p_r$. We now see that since $D^2/x^{1-\delta}\le p_1\le D^{1/3}$, we have
\begin{align*}
d_1&\le D_1\le \frac{N}{x^\delta},\\
N^2d_2d_3^2&\le N^2D_2D_3^2\le x^{2\delta}D^2/p_1<x^{1-\delta},\\
N^2d_1d_2^4d_3^3&\le N^2D_1D_2^4D_3^3\le x^{2\delta}D^3p_1<x^{2-\delta},\\
N d_1d_2^5d_3^2&\le ND_1D_2^5D_3^2\le x^\delta D^2 p_1^3<x^{2-\delta},
\end{align*}
so this factorization satisfies the conditions.

\textbf{Case 2: $p_2p_3\ge D^2/x^{1-3\delta}$.}

This is similar to the case above. Without loss of generality we may assume we are not in Case 1, so $p_1,p_2,p_3,p_4<D^2/x^{1-3\delta}$. We now set $D_1:=N x^{-2\delta}$, $d_2:=D_2:=p_2p_3$ and $D_3:=x^{2\delta} D/(Np_2p_3)$. Note that
\[
p_2p_3\le p_2^{1/3}(p_1p_2p_3^3)^{1/3}\le \frac{D^{2/3}}{x^{1/3-\delta}}D^{1/3}= \frac{D}{x^{1/3-\delta}}.
\]
In particular, $D_1,D_2,D_3\ge 1$ and we have that $D_1D_3=D/p_2p_3\ge x^{1/3-\delta}$. Thus $p_1^2<D^4/x^{2-6\delta}< x^{1/3-\delta}\le D_1D_3$, and $p_4^2\le D_1D_3/p_1$ since $p_1p_2p_3p_4^2\le D$. Moreover, for $j\ge 5$ we have $p_1\cdots p_{j-1}p_j^2\le D$, so $p_j^2\le D_1D_3/(p_1p_4\dots p_{j-1})$. We can greedily form products $d_1\le D_1$ and $d_3\le D_3$ out of $p_1p_4\cdots p_r$, by considering each prime in turn. We now see that since $D^2/x^{1-3\delta}\le p_2p_3<x^{1/4}$, we have
\begin{align*}
d_1&\le D_1\le\frac{N}{x^\delta},\\
N^2d_2d_3^2&\le N^2D_2D_3^2\le x^{2\delta} D^2/(p_2p_3)\le x^{1-\delta},\\
N^2d_1d_2^4d_3^3&\le N^2D_1D_2^4D_3^3\le x^{2\delta}D^3p_2p_3\le x^{2-\delta},\\
Nd_1d_2^5d_3^2&ND_1D_2^5D_3^2\le x^\delta D^2(p_2p_3)^3\le x^{2-\delta},
\end{align*}
so this gives a suitable factorization.

\textbf{Case 3: $p_1p_4\ge D^2/x^{1-3\delta}$.}

We may assume we are not in Case 1 or 2. In particular $\max(p_1,p_2p_3)< D^2/x^{1-3\delta}$, so $p_1p_4\le  p_1(p_2p_3)^{1/2}<D^3/x^{3/2-9\delta/2}<D/x^{1/3-\delta/2}$, and the argument is completely analogous to the case above, choosing $D_1:=N x^{-2\delta}$, $d_2:=D_2:=p_1p_4$ and $D_3:=x^{2\delta} D/(Np_1p_4)$, using the fact that $D^2/x^{1-3\delta}\le p_1p_4<x^{1/4}$.

\textbf{Case 4: $p_1p_4<D^2/x^{1-3\delta}$ and $p_2p_3<D^2/x^{1-3\delta}$. }

We set $D_1:=N x^{-\delta}$, $D_2:=D^2/x^{1-3\delta}$ and $D_3:=x^{1-2\delta}/(D N)$, noting that these are all at least 1. We see that one of $D_1$ or $D_3$ is also at least $D^2/x^{1-3\delta}$, since their product is $x^{1-3\delta}/D>x^{9/24}>D^4/x^{2-6\delta}$. We now wish to greedily form products $d_1\le D_1$, $d_2\le D_2$ and $d_3\le D_3$ by considering primes in turn. We start with $d_2=p_1p_4<D_2$ and either $d_1=1$ and $d_3=p_2p_3$ or $d_1=p_2p_3$ and $d_3=1$ depending on whether $p_2p_3>D_1$ or not. We now greedily form a sequence, where at the $j^{th}$ step we replace one of the $d_i$ with $d_ip_j$ provided $d_i p_j<D_i$ (the choice of $i\in\{1,2,3\}$ does not matter if there are multiple possibilities with $d_i p_j<D_i$), and we start with $j=5$. We stop if either we have included our final prime $p_r$ in one of the $d_i$, or there is a stage $j$ when $p_j d_1>D_1$, $p_j d_2>D_2$ and $p_j d_3>D_3$. If we stop because we have exhausted all our primes, then we see that we have found $d_1\le D_1$, $d_2\le D_2$ and $d_3\le D_3$ such that $d_1d_2d_3=p_1\cdots p_r$. It is then easy to verify that
\begin{align*}
d_1&\le D_1\le \frac{N}{x^\delta},\\
N^2d_2d_3^2&\le N^2D_2D_3^2\le x^{1-\delta},\\
N^2d_1d_2^4d_3^3&\le N^2 D_1D_2^4D_3^3\le \frac{D^5}{x^{1-5\delta}}<x^{2-\delta},\\
Nd_1d_2^5d_3^2&\le ND_1D_2^5D_3^2\le \frac{D^8}{x^{3-10\delta}}<x^{2-\delta}.
\end{align*}
Thus we just need to consider the situation when at some stage $j$ we have $p_j d_1>D_1$, $p_j d_2>D_2$ and $p_j d_3>D_3$. We see that this must first occur when $j$ is even, since for odd $j$ we have $p_j^3\le D/(p_1\cdots p_{j-1})=D_1D_2D_3/(d_1d_2d_3)$ and so $p_j\le \max(D_1/d_1,D_2/d_2,D_3/d_3)$. We must also have $j\ge 6$ since $j>4$ and is even. This implies $(p_j)^7\le p_1\cdots p_4p_5^3\le D$, so $p_j\le D^{1/7}\le x^{1/12-6\delta}$. 

We now set $d_2':=d_2p_j$ and $D_2':=D_2x^{1/12-6\delta}$, so that $D_2\le d_2'\le D_2'$. We set $D_3':=D_2 D_3/d_2'$. For all $\ell>j$ we have $p_{\ell}^2<D_1D_3'/(d_1d_3p_{j+1}\cdots p_{\ell-1})$, so we can greedily make products $d_1'\le D_1$ and $d_3'\le D_3'$ with $d_1'd_3'=d_1d_3p_{j+1}\cdots p_r$. In particular, we then have $d=d_1'd_2'd_3'$. We then verify
\begin{align*}
d_1'&\le D_1\le \frac{N}{x^\delta},\\
N^2 d_2'(d_3')^2&\le N^2d_2'(D_3')^2=\frac{N^2 D_2^2 D_3^2}{d_2'}\le N^2 D_2 D_3^2=x^{1-\delta},\\
N^2 d_1' (d_2')^4(d_3')^3&\le N^2 D_1 (d_2')^4(D_3')^3\le N^2D_1 D_2^4D_3^3 x^{1/12-6\delta}\le \frac{D^5}{x^{11/12+\delta}}<x^{2-\delta},\\
N d_1' (d_2')^5 (d_3')^2 &\le  N^2 D_1 (d_2')^5 (D_3')^2\le N^2D_1D_2^5D_3^2 x^{1/4-18\delta} \le \frac{D^8}{x^{11/4+8\delta}}\le x^{2-\delta}.
\end{align*}
We have now covered all cases, and so completed the proof of Proposition \ref{prpstn:Factorization}.
\end{proof}
%
%
%
%
%%%%%%%%%%%%%%%%%%%%%%%%%%%%%%%%%%%%%%%%%%%%%%%%%%%%%%%%%%
%
%
%
%
\begin{rmk}
By considering the situation when $N=x^{1/3}$, $p_1\approx p_2\approx  D^{2/7}$, $p_3\approx p_4\approx D^{1/7}$, and $p_j$ for $j\ge 5$ are small but satisfy $p_1\cdots p_r\approx D$, we see that Proposition \ref{prpstn:Factorization} cannot be extended to $D=x^{7/12+\delta}$ unless we impose further restrictions on $N$ or the $p_i$.
\end{rmk}
%
%
%
%
%%%%%%%%%%%%%%%%%%%%%%%%%%%%%%%%%%%%%%%%%%%%%%%%%%%%%%%%%%
%
%
%
%
\bibliographystyle{plain}
\bibliography{Bibliography}

\begin{thebibliography}{10}

\bibitem{Bombieri}
E.~Bombieri.
\newblock On the large sieve.
\newblock {\em Mathematika}, 12:201--225, 1965.

\bibitem{BFI1}
E.~Bombieri, J.~Friedlander, and H.~Iwaniec.
\newblock Primes in arithmetic progressions to large moduli.
\newblock {\em Acta Math.}, 156(3-4):203--251, 1986.

\bibitem{BFI2}
E.~Bombieri, J.~Friedlander, and H.~Iwaniec.
\newblock Primes in arithmetic progressions to large moduli. {II}.
\newblock {\em Math. Ann.}, 277(3):361--393, 1987.

\bibitem{Chen}
J.-R. Chen.
\newblock On the representation of a larger even integer as the sum of a prime
  and the product of at most two primes.
\newblock {\em Sci. Sinica}, 16:157--176, 1973.

\bibitem{DeshouillersIwaniec}
J.-M. Deshouillers and H.~Iwaniec.
\newblock Kloosterman sums and {F}ourier coefficients of cusp forms.
\newblock {\em Invent. Math.}, 70(2):219--288, 1982/83.

\bibitem{Drappeau}
S.~Drappeau.
\newblock Th\'{e}or\`emes de type {F}ouvry-{I}waniec pour les entiers friables.
\newblock {\em Compos. Math.}, 151(5):828--862, 2015.

\bibitem{ElliottHalberstam}
P.~D. T.~A. Elliott and H.~Halberstam.
\newblock A conjecture in prime number theory.
\newblock In {\em Symposia {M}athematica, {V}ol. {IV} ({INDAM}, {R}ome,
  1968/69)}, pages 59--72. Academic Press, London, 1970.

\bibitem{FouvryGrupp}
\'{E}. Fouvry and F.~Grupp.
\newblock On the switching principle in sieve theory.
\newblock {\em J. Reine Angew. Math.}, 370:101--126, 1986.

\bibitem{FouvryGrupp2}
\'{E}. Fouvry and F.~Grupp.
\newblock Weighted sieves and twin prime type equations.
\newblock {\em Duke Math. J.}, 58(3):731--748, 1989.

\bibitem{FouvryIwaniecDivisor}
\'{E}tienne Fouvry and Henryk Iwaniec.
\newblock The divisor function over arithmetic progressions.
\newblock {\em Acta Arith.}, 61(3):271--287, 1992.
\newblock With an appendix by Nicholas Katz.

\bibitem{Opera}
J.~Friedlander and H.~Iwaniec.
\newblock {\em Opera de cribro}, volume~57 of {\em American Mathematical
  Society Colloquium Publications}.
\newblock American Mathematical Society, Providence, RI, 2010.

\bibitem{Harman}
G.~Harman.
\newblock {\em Prime-detecting sieves}, volume~33 of {\em London Mathematical
  Society Monographs Series}.
\newblock Princeton University Press, Princeton, NJ, 2007.

\bibitem{HBVaughan}
D.~R. Heath-Brown.
\newblock Prime numbers in short intervals and a generalized {V}aughan
  identity.
\newblock {\em Canadian J. Math.}, 34(6):1365--1377, 1982.

\bibitem{IwaniecFactorable}
H.~Iwaniec.
\newblock A new form of the error term in the linear sieve.
\newblock {\em Acta Arith.}, 37:307--320, 1980.

\bibitem{May1}
J.~Maynard.
\newblock Primes in arithmetic progressions to large moduli {I}: {F}ixed
  residue classes.
\newblock preprint, \url{https://arxiv.org/abs/2006.06572}.

\bibitem{Shiu}
P.~Shiu.
\newblock A {B}run-{T}itchmarsh theorem for multiplicative functions.
\newblock {\em J. Reine Angew. Math.}, 313:161--170, 1980.

\bibitem{Vinogradov}
A.~I. Vinogradov.
\newblock The density hypothesis for {D}irichet {$L$}-series.
\newblock {\em Izv. Akad. Nauk SSSR Ser. Mat.}, 29:903--934, 1965.

\bibitem{Wu}
J.~Wu.
\newblock Chen's double sieve, {G}oldbach's conjecture and the twin prime
  problem.
\newblock {\em Acta Arith.}, 114(3):215--273, 2004.

\end{thebibliography}

\end{document}